\documentclass[11pt]{article}

\usepackage{enumerate}
\usepackage{amscd}
\usepackage{amsmath}
\usepackage{latexsym}
\usepackage{amsfonts}
\usepackage{amssymb}
\usepackage{amsthm}
\usepackage{verbatim}
\usepackage{mathrsfs}
\usepackage{enumerate}
\usepackage[colorlinks,linkcolor=red,anchorcolor=blue,citecolor=green]{hyperref}
\usepackage[numbers,sort&compress]{natbib}

\theoremstyle{plain}\newtheorem{definition}{Definition}[section]
\theoremstyle{definition}\newtheorem{theorem}{Theorem}[section]
\theoremstyle{plain}\newtheorem{lemma}[theorem]{Lemma}
\theoremstyle{plain}
\theoremstyle{plain}\newtheorem{prop}[theorem]{Proposition}
\theoremstyle{remark}\newtheorem{remark}{Remark}[section]
\usepackage{xcolor}

\newcommand{\wgr}[1]{\textcolor{black}{#1}}

\newcommand{\Div}{\mathrm{div}\,}
\newcommand{\B}{\Big}

\newcommand{\R}{\mathbb{R}}
\newcommand{\be}{\begin{equation}}
\newcommand{\ee}{\end{equation}}
 \newcommand{\ba}{\begin{aligned}}
 \newcommand{\ea}{\end{aligned}}

\newcommand{\fbxo}{\int_{\tilde{B}_{k}}\!\!\!\!\!\!\!\!\!\! -~\,}
\newcommand{\fbxozero}{\int_{\tilde{B}_{k_{0}}}\!\!\!\!\!\!\!\!\!\!\!\! -~\,}
\newcommand{\fbxozeroo}{\int_{B_{k_{0}}}\!\!\!\!\!\!\!\!\!\!\!\! -~\,}
\newcommand{\fbxoi}{\int_{\tilde{B}_{i}} \!\!\!\!\!\!\!\!\! -~\,}

\newcommand{\fbx}{\int_{B_{k}}\!\!\!\!\!\!\!\!\!\! -~}

\newcommand{\fqxoo}{\iint_{\tilde{Q}_{k}} \!\!\!\!\!\!\!\!\!\!\!\!\!\!-\hspace{-0.2cm}-\hspace{-0.3cm}-\,~}

\newcommand{\fqxoth}{\iint_{\tilde{Q}_{3}} \!\!\!\!\!\!\!\!\!\!\!\!\!\!-\hspace{-0.2cm}-\hspace{-0.3cm}-\,~}

\newcommand{\fqxol}{\iint_{\tilde{Q}_{l}} \!\!\!\!\!\!\!\!\!\!\!\!\!-\hspace{-0.2cm}-\hspace{-0.3cm}-\,~}
\newcommand{\fqxolo}{\iint_{Q_{l}} \!\!\!\!\!\!\!\!\!\!\!\!\!-\hspace{-0.2cm}-\hspace{-0.3cm}-\,~}
\newcommand{\fqxolor}{\iint_{Q(r)} \!\!\!\!\!\!\!\!\!\!\!\!\!\!\!\!\!-\hspace{-0.24cm}-\hspace{-0.3cm}-\,~\,\,\,\,}
\newcommand{\fqxolorr}{\iint_{\tilde{Q}(r)} \!\!\!\!\!\!\!\!\!\!\!\!\!\!\!\!\!-\hspace{-0.24cm}-\hspace{-0.3cm}-\,~\,\,\,\,}

\providecommand{\bysame}{\leavevmode\hbox to3em{\hrulefill}\thinspace}
  \newcommand{\f}{\frac}
    
  \newcommand{\ben}{\begin{enumerate}}
   \newcommand{\een}{\end{enumerate}}

\newcommand{\ti}{\nabla}

\newcommand{\Rmnum}[1]{\expandafter\@slowromancap\romannumeral #1@}

\allowdisplaybreaks

\numberwithin{equation}{section}
\begin{document}
\title{\wgr{A regularity criterion at one scale without pressure for suitable weak solutions to the Navier-Stokes equations}}
\author{Yanqing Wang\footnote{ Department of Mathematics and Information Science, Zhengzhou University of Light Industry, Zhengzhou, Henan  450002,  P. R. China Email: wangyanqing20056@gmail.com},\;~Gang Wu\footnote{School of Mathematical Sciences,  University of Chinese Academy of Sciences, Beijing 100049, P. R. China Email: \wgr{wugang2011@ucas.ac.cn}} ~~and ~\,
	Daoguo Zhou\footnote{
	College of Mathematics and Informatics, Henan Polytechnic University, Jiaozuo, Henan 454000, P. R. China Email:
	zhoudaoguo@gmail.com }}
\date{}
\maketitle
\begin{abstract}
In this paper, we continue our work in \cite{[JWZ]} to derive $\varepsilon$-regularity criteria \wgr{at one scale without pressure for suitable weak solutions to the Navier-Stokes equations}.
 We \wgr{establish an}
		$\varepsilon$-regularity criterion below of suitable weak solutions, for any  $\delta>0$,
$$
\iint_{Q(1)}|u|^{\f{5}{2}+\delta}dxdt\leq \varepsilon.
$$
As an application, we extend the previous corresponding results concerning the improvement of the classical Caffarelli--Kohn--Nirenberg theorem by a logarithmic factor.
 \end{abstract}
\section{Introduction}
\label{intro}
\setcounter{section}{1}\setcounter{equation}{0}

We focus on   the following   incompressible Navier-Stokes equations in   three-dimensional space
	\be\left\{\ba\label{NS}
	&u_{t} -\Delta  u+ u\cdot\ti
	u+\nabla \Pi=0, ~~\Div u=0,\\
	&u|_{t=0}=u_0,
	\ea\right.\ee
	where $u $ stands for the flow  velocity field, the scalar function $\Pi$ represents the   pressure.   The
	initial  velocity $u_0$ is divergence free.

One    important regularity criterion  of suitable weak solutions of \eqref{NS} is the following one due to \cite{[Lin],[LS]}: there exists an positive universal constant  $\varepsilon$ such that  	$u\in L^{\infty}(Q(1/2))$ provided   the following conditions is satisfied
\be\label{LLS}
\iint_{Q(1)}|u|^{3}+|\Pi|^{\f32}dxdt<\varepsilon.
\ee
This  $\varepsilon$-regularity criterion plays an important role in the study of the Navier-Stokes equations (see,
e.g., \cite{[ESS],[JS],[NRS],[WW2],[KY16],[RS2]} and the
references \wgr{therein)}.
\wgr{In \cite[page 8]{[SS18]}, the authors gave a comment on regular criterion \eqref{LLS}}:`` the bootstrapping enables to lower the exponent in the smallness condition from $3$ to $\f52+\delta$\wgr{(at the} cost of having to use smallness at all scales)."
Indeed, Gustafson,   Kang and   Tsai \wgr{in \cite{[GKT]} established the} following
$\varepsilon$-regularity criterion at all scales
\begin{align}
			\sup_{0<r \leq 1 }\,\,
			\iint_{Q(r)}|u|^{\f{5}{2}} dxdt \leq \varepsilon.
			\label{tsai1}\end{align}
We refer the reader to \cite{[GKT],[TX]} for \wgr{other kind of} $\varepsilon$-regularity criteria  in terms of the velocity, the vorticity, the
gradient of the vorticity  at all scale.
  Kukavica \cite{[Kukavica]} proposed three questions regarding regularity criterion \eqref{LLS}. In particular, the second issue
is that \wgr{whether  regularity criterion \eqref{LLS} holds for the exponent
less} than $3$. Recently, Guevara   and  Phuc \cite{[GP]} first answered this question  via establishing the regularity criteria below
		\be\label{GP}
		\|u\|_{L^{2p, 2q} (Q(1))}+\|\Pi\|_{L^{p, q}(Q(1))}< \varepsilon,
		~~~{3}/{ q}+{2}/{p}=
		7/2~~~\text{with}~1\leq q\leq2.\ee
Subsequently, in the spirit of \cite{[GP]}, the authors in      \cite{[HWZ]} further
generalized		Guevara and Phuc's results by proving \be\label{optical}
\|u\|_{L^{p,q}(Q(1))}+\|\Pi\|_{L^{1}(Q(1))}<\varepsilon,~~1\leq 2/p+3/q <2, 1\leq p,\,q\leq\infty.
\ee
 The third question posed by Kukavica  is that the pressure can be removed in \eqref{LLS}. In this direction,
  Wolf \cite{[Wolf10]} successfully   proved the following regularity criterion via introducing local pressure projection
$$
\iint_{Q(1)}|u|^{3}dxdt< \varepsilon.
$$

Furthermore, in \cite{[CW]}, \wgr{Chae and Wolf} studied Liouville type theorems for self-similar solutions to the Navier-Stokes
equations  by proving $\varepsilon$-regularity criteria
\be \label{wolfchae}
\sup_{-1\leq t\leq0}\int_{B(1)}|u|^{q}dx < \varepsilon,~~\f32<q\leq3.
\ee
Developing   the  technique as in  \cite{[Wolf10],[CW]}, the authors  in \cite{[JWZ]} obtained $\varepsilon$-regularity condition
\be\label{JWZ}
\iint_{Q(1)}|u|^{20/7}dxdt< \varepsilon.
\ee
Inspired by the comment to \eqref{LLS} mentioned above, Kukavica's questions  and  recent progress \eqref{GP}-\eqref{JWZ}, we try to prove the following result
 \begin{theorem}\label{the1.2}
		Let  the pair $(u,  \Pi)$ be a suitable weak solution to the 3D Navier-Stokes system \eqref{NS} in $Q(1)$. \wgr{For any  $\delta>0$, there exists an absolute positive constant $\varepsilon$
		such that if   $u$ satisfies} \be\label{wwz}
 \iint_{Q(1)}|u|^{\f{5}{2}+\delta}dxdt\leq \varepsilon,\ee
		then, $u\in L^{\infty}(Q(1/16)).$
	\end{theorem}

\begin{remark}
\wgr{This theorem is an improvement of corresponding results in \eqref{GP}-\eqref{wolfchae}.}
\end{remark}

An outline of \wgr{the proof for} Theorem 1.1 is as follows.
In the first step, following the path of \cite{[CW],[JWZ],[Wolf10]},  one establishes the Caccioppoli type inequalities just in terms of $u$ via the local energy inequality \eqref{wloc1}\wgr{(see Section 2 for more notations and details).}
 As mentioned in \cite{[Wolf10],[W15],[CW]}, the advantage of  local energy
inequality \eqref{wloc1}
  removed   the non-local effect of the pressure term. However,
 as stated in \cite{[JWZ]}, the cost of local energy inequality \eqref{wloc1} without
non-local pressure
is that the    velocity field $u$ is lack of  the  kinetic energy $\|u\|_{L^{\infty,2}}$. As observed in \cite{[JWZ]}, \wgr{$v=u+\nabla \Pi_{h}$} enjoys all the energy, namely, $\|v\|_{L^{\infty}L^{2}}$ and $\|v\|_{L^{2}L^{2}}$, where $\Pi_{h}$ is a harmonic function. However, since \wgr{$v$ depends on} the radius  $r$,  one can not construct iteration in terms of  $v$. This leads to the main difficulty in construction of Caccioppoli type inequality. The key point is the full application of $u=v-\nabla \Pi_{h}$ and
\wgr{absorbtion of} $v$ in the right hand side  by the left hand in local energy inequality \eqref{wloc1}.
Compared with the \wgr{proof in} \cite{[JWZ]}, we need to absorb  more than one $v$ in the right hand side  by the left hand in local energy inequality \eqref{wloc1}. After carefully choosing the suitable text function in \eqref{wloc1} together with the interpolation inequality \eqref{zw}, we can treat the term $\iint|v|^{3}\phi^{2\beta-1}$\wgr{(see Section 3 for more notations and details)}.
Since $\Pi_{2}$ meets $\Delta \Pi_{2}=-\wgr{\mathrm{divdiv}}(u\otimes u)$, a natural method is \wgr{using the} representation of $\Pi_{2} $ to bound $\iint  |\Pi_{2}\phi^{\beta-1}||v\phi^{\beta}| $.
Indeed, to the knowledge of the authors, it is worth remarking that we will utilize the representation of $\Pi_{2} $ with the test function rather than the pressure only used in all previous work. This guarantees that the representation of $\Pi_{2} $ and the integration domain in this term is consistent.  This is \wgr{of independent interest}.   Then we obtain the desired Caccioppoli type inequalities \eqref{wwzc}.

In the second step,
we  use Caccioppoli type inequalities \eqref{wwzc} and induction arguments \wgr{developed in} \cite{[CKN],[TY],[RWW],[JWZ],[CW]}  to \wgr{prove Theorem 1.1}.
In contrast with \wgr{the previous argument},
 it should be noted that there exist two difficult terms to bound   in local energy inequality \eqref{wloc1}  since the     integration of time in \eqref{wwz}   is just $\f{5+2\delta}{2}$\wgr{(see \eqref{loc21} for more details)}. The first one is
 $\int^{t}_{-r^{2}_{k_{0}}}\int_{B_{1}}\Gamma\phi  v\otimes\nabla\Pi_{h}:\nabla^{2}\Pi_{h}   $, which will be bounded by the  introduction of  new quantity in \eqref{goal} in induction arguments. To control the second one $\int^{t}_{-r^{2}_{k_{0}}}\int_{B_{1}}  \Pi_{2}v\cdot\nabla(\Gamma\phi )$,  we decompose $\Pi_2$ into \wgr{two parts} (see \eqref{2.4.10}-\eqref{2.4.11}) and \wgr{make use of} Lemma \ref{CW}
to obtain the desired estimates.

Next we turn attention to the Caccioppoli type inequality,
For \wgr{the reader's 	convenience, before} we formulate our  proposition,  we recall the \wgr{known results} proved in \cite{[W15],[CW],[JWZ]}, respectively,
  \begin{align}
 &\|u\|^{2}_{L^{3,\f{18}{5}}Q(\f{1}{2})}+ \|\nabla u\|^{2}_{L^{2}(Q(\f{1}{2}))}
 \leq
  C \|  u\|^{2}_{L^{3}(Q(1))}+C\|  u\|^{3}_{L^{3}(Q(1))}\wgr{,}
   \label{Wolf10}\\
   \label{cw}
 &\|u\|^{2}_{L^{3,\f{18}{5}}Q(\f{1}{2})}+  \|\nabla u\|^{2}_{L^{2}(Q(\f{1}{2}))}
 \leq
  C \|  u\|^{2}_{L^{\f{3q}{2q-3},q}(Q(1))}+C\|  u\|^{\f{3q}{2q-3}}_{L^{\f{3q}{2q-3},q}(Q(1))},~~~\f32<q\leq3\wgr{,}\\
  & \|u\|^{2}_{L^{\f{20}{7},\f{15}{4}}Q(\f{1}{2})}+ \|\nabla u\|^{2}_{L^{2}(Q(\f{1}{2}))}
 \leq
  C \|  u\|^{2}_{L^{\f{20}{7}}(Q(1))}+C\|  u\|^{4}_{L^{\f{20}{7}}(Q(1))}. \label{jwz}
 \end{align}

\begin{prop}\label{prop1.2}
Assume that $u$ is a   suitable weak \wgr{solution to the Navier-Stokes equations, $3/2\leq2/p+3/q<2$} with $p\geq2,q\geq 12/5$. \wgr{Then we have, for any $R>0$,}
\be\ba
&\|u\|^{2}_{L^{2,6}(Q(\f{R}{2}))}+ \|\nabla u\|^{2}_{L^{2}(Q(\f{R}{2}))} \\
\leq &
  CR^{\f{3\alpha-4}{\alpha}}\|  u\|^{2}_{L^{p,q}(Q(R))}
  +CR^{ \f{5\alpha-8}{\alpha}}\|  u\|^{4}_{L^{ p,q} (Q(R))}
  +CR^{\f{3\alpha-5}{\alpha-1}}\|  u\|^{\f{2\alpha}{\alpha-1}}_{L^{p,q}(Q(R))}\wgr{,}
  \label{wwzc}
  \ea\ee
\wgr{where $\alpha=\frac{2}{\frac{2}{p}+\frac{3}{q}}$.}
 \end{prop}

 We present two application of new   $\varepsilon$-regularity criterion \eqref{wwz} \wgr{without pressure} at one scale.
This is in part motivated  by recent works  \cite{[RWW18],[WW18]}, where the authors found that at one scale it is useful to establish  new $\varepsilon$-regularity criterion to  obtain better box dimension and the improvement of  the classical Caffarelli-Kohn-Nirenberg theorem by a logarithmic factor.
Box dimension (Minkowski dimension)
is
a widely used fractal dimension\wgr{(see Section 5 for the definition)}.
The relationship between Hausdorff dimension and   box dimension is that the   former is less than the latter (see e.g. \cite{[Falconer]}). More \wgr{information} on     box  dimension can
be found in  \cite{[Falconer]}. 
Making use of Theorem 1.1 and following the path of \cite{[KY16],[WW2]}, one can derive that the (upper) box dimension of \wgr{the singular points set} $\mathcal{S}$ is at most $37/30(\approx1.23).$  This improves the previous box dimension of   $\mathcal{S}$ obtained in \cite{[KP],[RS2],[Kukavica],[KY16],[WW2]}. We leave the proof to the interested author. Indeed, the proof   is simple than that of \cite{[WW2],[HWZ],[RWW],[KY16]} owing to $\varepsilon$-regularity criterion \eqref{wwz} without pressure holds at one scale.
Finally, we are concerned with the improvement  of the
 Caffarelli-Kohn-Nirenberg theorem by a logarithmic factor.
 In  \cite{[CL]}, Choe and Lewis introduced the generalized Hausdorff measure  $\Lambda(\mathcal{S},r(\log(e/r))^{\sigma})$  (for the detail, \wgr{see Section 6)}    and proved that \be\label{cl}\Lambda(\mathcal{S},r(\log(e/r))^{\sigma})=0  (0\leq\sigma<3/44).\ee
 \eqref{cl} with $\sigma=0$ reduces to
the celebrated
Caffarelli-Kohn-Nirenberg theorem for the three-dimensional time-dependent Navier-Stokes system. Recently, there are some efforts to improve the bound of $\sigma$ in \eqref{cl}.
  $\sigma$ is bounded by  $1/6$ by Choe and Yang in  \cite{[CY1]}. Later,
Ren, Wang and  Wu \wgr{\cite{[RWW18]}} improved  the bound of $\sigma$ to $27/113$.
Inspired by   the new $\varepsilon$-regularity criterion \eqref{wwz}, we have the following result
\begin{theorem}\label{the1.5}
 Let $\mathcal{S}  $ stand   for  the set of all the potential interior singular \wgr{points} of suitable weak solutions  to \eqref{NS} and $0\leq\sigma<4/11$. There holds
$$\Lambda(\mathcal{S},r(\log(e/r))^\sigma)=0.$$
\end{theorem}
\begin{remark}
Theorem \ref{the1.5} is an improvement of the known  corresponding results in  \wgr{\cite{[CL],[CY1],[RWW18]}}.
\end{remark}

The remainder of this paper  is structured as follows. \wgr{In Section} \ref{sec2}, we start with \wgr{the details} of   Wolf's the local pressure projection $\mathcal{W}_{p,\Omega}$ and recall the definition of local suitable weak solutions due to \cite{[Wolf10],[W15]}. Then, we   establish \wgr{some auxiliary lemmas}.
 The Caccioppoli type inequality \eqref{wwzc} is derived in
  Section \ref{sec3}.
 Section \ref{sec4} is devoted to the proof of Theorem \ref{the1.2}.
 Finally, we consider the \wgr{improvements on} Caffarelli--Kohn--Nirenberg theorem by a logarithmic factor
in Section \ref{sec6}.

{\bf Notations:}
Throughout this paper, we denote
\begin{align*}
     &B(x,\mu):=\{\wgr{y\in \mathbb{R}^{3}}||x-y|\leq \mu\}, && B(\mu):= B(0,\mu), && \tilde{B}(\mu):=B(x_{0},\,\mu),\\
        &Q(x,t,\mu):=B(x,\,\mu)\times(\wgr{t-\mu^{2}}, t),  && Q(\mu):= Q(0,0,\mu), && \tilde{Q}(\mu):= Q(x_{0},t_{0},\mu),\\
          &r_{k}=2^{-k},\quad &&\tilde{B}_{k}:= \tilde{B}(r_{k}), \quad ~~ &&\tilde{Q}_{k}:=\tilde{Q}(r_{k}).
\end{align*}
For $p\in [1,\,\infty]$, the notation $L^{p}((0,\,T);X)$ stands for the set of measurable functions on the interval $(0,\,T)$ with values in $X$ and $\|f(t,\cdot)\|_{X}$ belongs to $L^{p}(0,\,T)$.
  For simplicity,   we write $$\|f\| _{L^{p,q}(Q(r))}:=\|f\| _{L^{p}(-r^{2},0;L^{q}(B(r)))}~~   \text{and}~~
  \|f\| _{L^{p}(Q(r))}:=\|f\| _{L^{p,p}(Q(r))}\wgr{.} $$
\wgr{We also denote
\begin{equation*}
\begin{split}
  E(\mu)=\mu^{-1}\|u\|_{L^{\infty,2}(Q(\mu))}^{2}, &\qquad E_{*}(\mu)=\mu^{-1}\|\nabla u\|_{L^{2}(Q(\mu))}^{2},\\
  E_{p}(\mu)=\mu^{p-5}\|u\|_{L^{p}(Q(\mu))}^{p}, &\qquad J_{p}(\mu)=\mu^{2p-5}\|\nabla u\|_{L^{p}(Q(\mu))}^{p}.
\end{split}
\end{equation*}
}
\wgr{In addition, denote}
  the average of $f$ on the set $\Omega$ by
  $\overline{f}_{\Omega}$. For convenience,
  $\overline{f}_{r}$ represents  $\overline{f}_{B(r)}$ and $\overline{\Pi}_{\tilde{B}_{k}}$
  is denoted by $\tilde{\Pi}_{k}$.
  $|\Omega|$ represents the Lebesgue measure of the set $\Omega$.
  For exponent $p\in [1,\infty)$, we define the H\"older conjugate $p^{\ast}$ through the relation  $1/p^{\ast}=1-1/p$.
  We will use the summation convention on repeated indices.
 $C$ is an absolute constant which may be different from line to line unless otherwise stated in this paper.

\section{Preliminaries}\label{sec2}

We begin with \wgr{the Wolf's} local pressure projection $\mathcal{W}_{p,\Omega}:$ $W^{-1,p}(\Omega)\rightarrow W^{-1,p}(\Omega)$ $(1<p<\infty)$.
 More precisely, for any  $f\in W^{-1,p}(\Omega)$, we define \wgr{$\mathcal{W}_{p,\Omega}(f)= \nabla\Pi$}, where $\Pi$ satisfies \eqref{GMS}.
Let $\Omega$  be a  bounded domain with $\partial\Omega\in C^{1}$.
According to the $L^p$ theorem of Stokes system in \cite[Theorem 2.1, p149]{[GSS]},
there exists a unique pair $(u,\Pi)\in W^{1,p}(\Omega)\times L^{p}(\Omega)$ such that
\be\label{GMS}
-\Delta u+\nabla\Pi=f,~~ \text{div}\,u=0, ~~u|_{\partial\Omega}=0,~~ \int_{\Omega}\Pi dx=0.
\ee
Moreover, this pair is subject to the inequality
$$
\|u\|_{\wgr{W^{1,p}}(\Omega)}+\|\Pi\|_{\wgr{L^p}(\Omega)}\leq C\|f\|_{\wgr{W^{-1,p}}(\Omega)}.
$$
Let $\nabla\Pi= \mathcal{W}_{p,\Omega}(f)$ $(f\in L^p(\Omega))$, then $\|  \Pi\|_{L^p(\Omega)}\leq C\|f\|_{L^p(\Omega)},$ where we used the fact that $L^{p}(\Omega)\hookrightarrow W^{-1,p}(\Omega)$.  Moreover, from $\Delta \Pi=\text{div}\,f$, we see that $\|  \nabla\Pi\|_{L^p(\Omega)}\leq C(\|f\|_{L^p(\Omega)}+ \|  \Pi\|_{L^p(\Omega)}) \leq C\|f\|_{L^p(\Omega)}.$
Now, we  present the definition of suitable weak solutions of Navier-Stokes equations \eqref{NS}.
	\begin{definition}\label{defi}
		A  pair   $(u, \,\Pi)$  is called a suitable weak solution to the Navier-Stokes equations \eqref{NS} provided the following conditions are satisfied,
		\begin{enumerate}[(1)]
			\item $u \in L^{\infty}(-T,\,0;\,L^{2}(\mathbb{R}^{3}))\cap L^{2}(-T,\,0;\,\dot{H}^{1}(\mathbb{R}^{3})),\,\Pi\in
			L^{3/2}(-T,\,0;L^{3/2}(\mathbb{R}^{3}));$\label{SWS1}
			\item$(u, ~\Pi)$~solves (\ref{NS}) in $\mathbb{R}^{3}\times (-T,\,0) $ in the sense of distributions;\label{SWS2}
			\item The local energy inequality reads,
for a.e. $t\in[-T,0]$ and non-negative function $\phi(x,s)\in C_{0}^{\infty}(\mathbb{R}^{3}\times (-T,0) )$,
			 \begin{align}
  &\int_{B(r)}|v|^2\phi (x,t)  d  x+ \int^{t}_{-T }\int_{B(r)}\big|\nabla v\big|^2\wgr{\phi (x,s) dx ds}\nonumber\\  \leq&   \int^{t}_{-T }\int_{B(r)} | v |^2(  \Delta \phi +  \partial_{t}\phi )  d  x d s +\int^{t}_{-T }\int_{B(r)}|v|^{2}u\cdot\nabla \phi    dsds\nonumber\\
& +\int^{t}_{-T }\int_{B(r)} \phi ( u\otimes v :\nabla^{2}\Pi_{h} ) \wgr{dxds}   +\int^{t}_{-T }\int_{B(r)} \phi \Pi_{1}v\cdot\nabla \phi   dxds+\int^{t}_{-T }\int_{B(r)} \phi \Pi_{2}v\cdot\nabla \phi   dxds.\label{wloc1}
 \end{align}
\wgr{Here, $\nabla\Pi_h=-\mathcal{W}_{p,B(R)}(u)$, $\nabla\Pi_1=\mathcal{W}_{p,B(R)}(\Delta u)$, $\nabla\Pi_2=-\mathcal{W}_{p,B(R)}(u\cdot\nabla u)$, $v=u+\nabla\Pi_h$.}
 In addition, $ \nabla\Pi_{h}, \nabla\Pi_{1}$ and $\nabla\Pi_{2}$ meet the following \wgr{facts}
	\begin{align}   &\|\nabla\Pi_{h}\|_{L^p(B(R))}\leq  \wgr{C}\|u\|_{L^p(B(R))}, \label{ph}\\
 &\|\wgr{\Pi_{1}}\|_{L^2(B(R))}\leq \wgr{C} \|\nabla u\|_{L^2(B(R))},\label{p1}\\
 &\|\wgr{\Pi_{2}}\|_{L^{p/2}(B(R))}\leq \wgr{C} \| |u|^{2}\|_{L^{p/2}(B(R))}.\label{p2}
\end{align}	\end{enumerate}
	\end{definition}

We list some
interior estimates
of  harmonic functions $\Delta h=0$, which will be frequently utilized later. Let $1\leq p,q\leq\infty$ and \wgr{$0<r<\rho$}, then, it holds
\be\label{h1}\|\nabla^{k}h\|_{L^{q}
(B(r))}\leq \f{Cr^{\f{3}{q}}}{(\rho-r)^{\f{3}{p}+k}}\|h\|_{L^{p}(B(\rho))}\wgr{,}\ee
\be\label{h2}
 \| h-\overline{h}_{r}\|_{L^{q}
(B(r))}\leq \f{Cr^{\f{3}{q}+1}}{(\rho-r)^{\f{3}{q}+1 }}\|h-\overline{h}_{\rho}\|_{L^{q}(B(\rho))}.\ee
 The proof of \eqref{h1} rests on the mean value property of harmonic functions. This together with
  mean value theorem leads to \eqref{h2}. We leave the \wgr{details} to the reader.

\wgr{For reader's convenience}, we recall an interpolation inequality.
For each $2\leq l\leq\infty$ and $2\leq k\leq6$ satisfying $\f{2}{l}+\f{3}{k}=\f{3}{2}$, according to the H\"{o}lder   inequality  and the Young  inequality, we know that
\begin{align}
\|u\|_{\wgr{L^{l,k}}(Q(\mu))}&\leq C\|u\|_{\wgr{L^{\infty,2}}(Q(\mu))}^{1-\f {2} {l}}\|u\|_{\wgr{L^{2,6}}(Q(\mu))}^{\f {2} {l}}\nonumber\\
&\leq C\|u\|_{\wgr{L^{\infty,2}}(Q(\mu))}^{1-\f {2} {l}}(\|u\|_{\wgr{L^{\infty,2}}(Q(\mu))}
+\|\nabla u\|_{L^2(Q(\mu))})^{\f {2} {l}}\nonumber\\
&\leq C (\|u\|_{\wgr{L^{\infty,2}}(Q(\mu))}
+\|\nabla u\|_{L^2(Q(\mu))}).\label{sampleinterplation}
\end{align}

\begin{lemma}\label{zc} Let $\wgr{1\leq 2/p+3/q <2}, 1\leq p,\,q\leq\infty $ and \wgr{$\alpha=\frac{2}{\frac{2}{p}+\frac{3}{q}}$}. There is an absolute constant $C$   such that
\be\ba\label{zw}
\|u\|_{L^{3}(Q(\rho))}^{3} \leq C \rho^{3(\alpha-1)/2} \|u\|_{L^{p,q}(Q(\rho))}^{\alpha}\Big(\|u\|_{\wgr{L^{\infty,2}}(Q(\rho))}^{2}+\|\nabla u\|_{L^{2}(Q(\rho))}^{2}\Big)^{(3-\alpha)/2}.
\ea\ee
\end{lemma}

\begin{remark}
Lemma \ref{zc} is obtained in \cite{[HWZ]}. Here we present a different proof from that in \cite{[HWZ]}. New proof \wgr{allows one to apply} it to more general case.
\end{remark}
\begin{proof}
We denote
$$m=(3-\alpha)(\f{p}{\alpha})^{\ast},n=(3-\alpha)(\f{q}{\alpha})^{\ast}\wgr{.}$$
Thanks to the H\"older inequality, we find that
$$\ba
\|u^{\alpha} u^{2-\alpha} \|_{L^{m^{\ast},
n^{\ast}}}&\leq \|u^{\alpha}\|_{L^{\f{p}{\alpha},\f{q}{\alpha}}}
\|u^{2-\alpha} \|_{L^{\f{m}{2-\alpha},\f{n}{2-\alpha}}}\\
&\leq \wgr{\|u\|^{\alpha}_{L^{p,q}}
\|u \|^{2-\alpha}_{L^{m,n}}.}
\ea$$
Using the latter inequality, the H\"older inequality,
and \eqref{sampleinterplation},
 we infer that
\begin{align}
\|u\|^{3}_{L^{3}}
&\leq   \|u \|_{L^{m,
n}(Q(\rho))}
\|u^{\alpha} u^{2-\alpha}    \|_{L^{m^{\ast},
n^{\ast}}}
  \nonumber\\&\leq   \|u \|^{3-\alpha}_
  {L^{m,n}(Q(\rho))}
\|u \|^{\alpha}_{L^{p,q}}\nonumber\\
&\leq C\rho^{3(\alpha-1)/2}\|u\|_{L^{p,q}(Q(\rho))}^{\alpha}\|u\|^{3-\alpha}
_{L^{2(\f{p}{\alpha})^{\ast},2(\f{q}{\alpha})^{\ast}}(Q(\rho))}
\nonumber\\
&\leq C \rho^{3(\alpha-1)/2} \|u\|_{L^{p,q}(Q(\rho))}^{\alpha}\Big(\|u\|_{\wgr{L^{\infty,2}}(Q(\rho))}^{2}+\|\nabla u\|_{L^{2}(Q(\rho))}^{2}\Big)^{(3-\alpha)/2}.\label{2.2.20}
\end{align}
This completes                                     the proof.
\end{proof}
In additon, we recall two well-known iteration lemmas.		
	
\begin{lemma}\label{iter1}\cite[Lemma V.3.1,   p.161]{[Giaquinta]}
			Let $I(s)$ be a bounded nonnegative function in the interval $[r, R]$. Assume that for every $\sigma, \rho\in [r, R]$ and  $\sigma<\rho$ we have			$$I(\sigma)\leq A_{1}(\rho-\sigma)^{-\alpha_{1}} +A_{2}(\rho-\sigma)^{-\alpha_{2}} +A_{3}+ \ell I(\rho)$$
			for some non-negative constants  $A_{1}, A_{2}, A_{3}$, non-negative exponents $\alpha_{1}\geq\alpha_{2}$ and a parameter $\ell\in [0,1)$. Then there holds
			$$I(r)\leq c(\alpha_{1}, \ell) [A_{1}(R-r)^{-\alpha_{1}} +A_{2}(R-r)^{-\alpha_{2}} +A_{3}].$$
		\end{lemma}

The following lemma is a generalization of corresponding result in \cite{[CW]} (see \cite[Lemma  2.9,   p.558]{[CW]}).
  \begin{lemma}\label{CW}
Let $f\in L^{q}(Q(1))$ with $\f{3}{\tau-1}>q>1$ and \wgr{$0<r_{0}<1/2$}. Suppose that   for all $(x_{0},t_{0})\in Q(1/2)$ and $r_{0}\leq r\leq \f12$
\be\label{lem2.31}\|f-\overline{f}_{\tilde{B}(r)}\|_{L^{p,q}(\tilde{Q}(r))} \leq C r^{\tau}.\ee
Let $\nabla\Pi=\wgr{\mathcal{W}_{q,{B}(1)}} (\nabla\cdot f)$. Then for all $(x_{0},t_{0})\in Q(1/2)$ and
$r_{0}\leq r \leq\f12$, it holds
$$\|\Pi-\overline{\Pi}_{\tilde{B}(r)}\|_{L^{p,q}(\tilde{Q}(r))} \leq C r^{\tau}.$$
\end{lemma}
\begin{proof}
In view of  the definition of  pressure projection $\mathcal{W}_{q,B(1)}$, we know that
\be\label{2.8}
\|\Pi\|_{L^{q}(B(1))}\leq C\|f-\overline{f}_{B(1)}\|_{L^{q}(B(1))}.
\ee
We introduce a cut-off function $\phi(x)$ such that $\phi(x)=1, x\in \tilde{B}(\f{3r}{4}), \phi(x)=0, x\in \tilde{B}^{c}(r).$\\
Note that $$\Delta \Pi=\text{div}\wgr{\mathcal{W}_{q,B(1)}}(\nabla\cdot f).$$
 We split $\Pi$ into two part $\Pi=\Pi_{(1)}+\Pi_{(2)}$, where
  $$\Delta \Pi_{(1)}=-\text{div}\wgr{\mathcal{W}_{q,B(1)}}(\nabla\cdot [\phi( f-\overline{f}_{\tilde{B}(r)})]),$$
  which follows from that
  $$\Delta \Pi_{(2)}=0, x\in \tilde{B}(3r/4).$$
Thanks  classical Calder\'on-Zygmund theorem, we have
$$
\|\Pi_{(1)}-\overline{\Pi_{(1)}}_{\tilde{B}(r)}\|_{L^{q}(\tilde{B}(r))}\leq C\|f-\overline{f}_{\tilde{B}(r)}\|_{L^{q}(\tilde{B}(r))}.
$$
This and  hypothesis \eqref{lem2.31} yield
\begin{equation}\label{cz1}
\|\Pi_{(1)}-\overline{\Pi_{(1)}}_{\tilde{B}(r)}\|_{L^{p,q}(\tilde{Q}(r))}\leq \wgr{Cr^{\tau}}.
\end{equation}
\wgr{The interior} estimates of harmonic functions \eqref{h2} and \wgr{the triangle} inequality guarantee that, for $ \theta<1/2$, we have
$$\ba
&\int_{\tilde{B}(\theta r)}|\Pi_{(2)}-\overline{\Pi_{(2)}}_{\tilde{B}(\theta r)}|^{q}dx\\
\leq& \f{C( r\theta )^{3+q}}{(\f{r}{2})^{3+q}}
\int_{\tilde{B}(  r/2)}|\Pi_{(2)}-\overline{\Pi_{(2)}}_{\tilde{B}(r/2)}|^{q}dx\\
\leq& C\theta^{3+q}
\int_{\tilde{B}(  r/2)}|\Pi-\overline{\Pi}_{ \tilde{B}( r/2)}|^{q}dx+
\int_{\tilde{B}(  r/2)}|\Pi_{(1)}-\overline{\Pi_{(1)}}_{\tilde{B}(r/2)}|^{q}dx.
\ea$$
We derive from the latter inequality   and  \eqref{cz1}   that
$$\ba
\|\Pi_{(2)}-\overline{\Pi_{(2)}}_{\tilde{B}(\theta r)}\|_{L^{p,q}(Q(\theta r))}
 \leq C\theta^{\f3 q+1}
 \|\Pi-\overline{\Pi}_{ \tilde{B}( r/2)}\|_{L^{p,q}(Q( r/2))} +C r^{\tau}.
\ea$$
 With the help of the triangle inequality again, \eqref{cz1} and the last inequality, we infer that
$$\ba&\|\Pi-\overline{\Pi}_{\tilde{B}(\theta r)}\|_{L^{p,q}(Q(\theta r))}\\\leq&
\|\Pi_{(1)}-\overline{\Pi_{(1)}}_{\tilde{B}(\theta r)}\|_{L^{p,q}(Q(\theta r))} +\|\Pi_{(2)}-\overline{\Pi_{(2)}}_{\tilde{B}(\theta r)}\|_{L^{p,q}(Q(\theta r))}
\\
\leq& \wgr{C r^{\tau}+C\theta^{\f3 q+1}
 \|\Pi-\overline{\Pi}_{ \tilde{B}( r/2)}\|_{L^{p,q}(Q( r/2))} +C r^{\tau}}\\
\leq& C\theta^{\f3 q+1}
 \|\Pi-\overline{\Pi}_{ \tilde{B}( r)}\|_{L^{p,q}(Q( r))} +C r^{\tau},
\ea$$
where we used the fact that $\|g-\overline{g}_{B(r)}\|_{\wgr{L^l}(B(r))}\leq C\|g-c\|_{\wgr{L^l}(B(r))}$ with \wgr{$l\geq1$}.

Now, thanks to $\f3 q+1>\tau$, invoking iteration \wgr{Lemma 2.2} and \eqref{2.8}, we see that
 $$\ba\|\Pi-\overline{\Pi}_{\tilde{B}(  r)}\|_{L^{p,q}(Q( r))}
&\leq  Cr^{\tau}\|\Pi-\overline{\Pi}_{\tilde{B}(\wgr{1/2})}\|_{L^{p,q}(Q(\wgr{1/2}))}+Cr^{\tau}\\
&\leq  Cr^{\tau}\|f-\overline{f}_{B(1)}\|_{L^{p,q}(Q(  1))}
+Cr^{\tau}
\\&\leq Cr^{\tau}.
\ea$$
This completes the proof of this lemma.
\end{proof}

To prove Theorems \ref{the1.5}, we need the following result.

\begin{prop}\wgr{\cite{[RWW18]}}\label{add}
 Let $\mathcal{S}  $ stand   for  the set of all the potential interior singular  points  of suitable weak solutions  to \eqref{NS} and $\tau $ \wgr{be defined in Lemma 6.2}. Then, there holds, for $0\leq\sigma<\f{1}{\tau+1}$,
$$\Lambda(\mathcal{S},r(\log(e/r))^\sigma)=0.$$
\end{prop}
\section{Caccioppoli estimate}
		\label{sec3}
		\setcounter{section}{3}\setcounter{equation}{0}
This section contains the proof of Proposition  \ref{prop1.2}.
Proposition  \ref{prop1.2}  turns out to be a corollary of the following proposition.

\begin{proof}[\wgr{Proof of Proposition  \ref{prop1.2}}]Consider $0<R/2\leq r<\f{3r+\rho}{4}<\f{r+\rho}{2}<\rho\leq R$. Let $\phi(x,t)$ be non-negative smooth function supported in $Q(\f{r+\rho}{2})$ such that
$\phi(x,t)\equiv1$ on $Q(\f{3r+\rho}{4})$,
$|\nabla \phi| \leq  C/(\rho-r) $ and $
|\nabla^{2}\phi|+|\partial_{t}\phi|\leq  C/(\rho-r)^{2} .$

Let $\nabla\Pi_{h}=\mathcal{W}_{q,B(\rho)}(u)$, then, there holds
\begin{align}
&\|\nabla \Pi_{h}\|_{L^{p,q}(Q(\rho))}\leq C\|u\|_{L^{p,q}(Q(\rho))},\label{wp1}\\
 &\|  \Pi_{1}\|_{L^{2}(Q(\rho))}\leq C\|\nabla u\|_{L^{2}(Q(\rho))},\label{wp2}\\
 &\|  \Pi_{2}\|_{L^{\f{p}{2},\f{q}{2}}(Q(\rho))}\leq C\|  |u|^{2}\|_{L^{\f{p}{2},\f{q}{2}}(Q(\rho))}.\label{wp3}
 \end{align}
By virtue of interior estimate of harmonic function \eqref{h1} and \eqref{wp1}, we conclude that
\begin{equation}\label{2.3.8}
\begin{split}
&\|\nabla \Pi_{h} \|_{L^{2,\infty}(Q(\f{r+\rho}{2}))}\\
\leq & \f{C}{(\rho-r)^{\f32}}\| \nabla\Pi_{h} \|_{L^{2,2}(Q( \rho ))}\leq \f{C\rho^{\f{5\alpha-4}{2\alpha}}}{(\rho-r)^{\f32}}\| \nabla\Pi_{h} \|_{L^{p,q}(Q( \rho ))}
\wgr{\leq \f{C\rho^{\f{5\alpha-4}{2\alpha}}}{(\rho-r)^{\f32}}\|u\|_{L^{p,q}(Q( \rho ))},}
\end{split}
\end{equation}
\begin{equation}\label{2.3.9}
\begin{split}
&\|\nabla^{2}\Pi_{h} \|_{L^{2,\infty}(Q(\f{r+\rho}{2}))}\\
\leq &\f{C}{(\rho-r)^{\f32+1}}\| \nabla\Pi_{h} \|_{L^{2,2}(Q( \rho ))}\leq \f{C\rho^{\f{5\alpha-4}{2\alpha}}}{(\rho-r)^{\f52}}\| \nabla\Pi_{h} \|_{L^{p,q}(Q( \rho ))}
\wgr{\leq \f{C\rho^{\f{5\alpha-4}{2\alpha}}}{(\rho-r)^{\f52}}\|u\|_{L^{p,q}(Q( \rho ))},}
\end{split}
\end{equation}
\begin{equation}
\begin{split}
&\|\nabla \Pi_{h} \|_{L^{2,4}(Q(\f{r+\rho}{2}))}\\
\leq & \f{\wgr{C}\rho^{\f34}}{(\rho-r)^{\f32}}\|\nabla \Pi_{h}\|_{L^{2,2}\wgr{(Q( \rho ))}}
\leq \f{C\rho^{\f{13\alpha-8}{4\alpha}}}{(\rho-r)^{\f32}}\| \nabla\Pi_{h} \|_{L^{p,q}(Q( \rho ))}
\wgr{\leq \f{C\rho^{\f{13\alpha-8}{4\alpha}}}{(\rho-r)^{\f32}}\|u\|_{L^{p,q}(Q( \rho ))}.}
\end{split}
\end{equation}

We define $\beta=\f{1}{\alpha-1}$ \wgr{and} choose $\phi^{2\beta}$ as the
non-negative function in the local energy inequality  to get
 \begin{align}
  &\int_{B(\wgr{\f{r+\rho}{2}})}|v|^2\phi^{2\beta} (x,t)  d  x+ \iint_{Q(\wgr{\f{r+\rho}{2}})}\big|\nabla v\big|^2\phi^{2\beta} (x,s) d  x ds  \nonumber\\  \leq&   \f{C}{(\rho-r)^{2}}\iint_{Q(\f{r+\rho}{2})} |v|^{2}  + \f{C}{(\rho-r)}\iint_{Q(\f{r+\rho}{2})}|v|^{2}|u | \phi^{2\beta-1}    \nonumber\\
& +\iint_{Q(\f{r+\rho}{2})}\phi^{2\beta}| u\otimes v :\nabla^{2}\Pi_{h} |     +\f{C}{(\rho-r)}\iint_{Q(\f{r+\rho}{2})} |\Pi_{1}v| \phi^{2\beta-1}   \nonumber\\& +\f{C}{(\rho-r)}\iint_{Q(\f{r+\rho}{2})}  |\Pi_{2}\phi^{\beta-1}||v\phi^{\beta}| \nonumber\\
\wgr{=:}&I+II+III+IV+V   .\label{wloc13}
 \end{align}

By means of the triangle inequality, \eqref{wp1} and the H\"older inequality, we see that
\be
\iint_{Q(\f{r+\rho}{2})} |v|^{2}\leq\iint_{Q(  \rho )} |u|^{2}+|\nabla\Pi_{h}|^{2}\leq C\rho^{\f{5\alpha-4}{ \alpha}} \| u \|^{2}_{L^{p,q}(Q( \rho ))},\label{3.2}\ee
which implies that
$$
I\leq \f{C\rho^{\f{5\alpha-4}{ \alpha}} }{(\rho-r)^{2}}\| u \|^{2}_{L^{p,q}(Q( \rho ))}.
$$
\wgr{It is obvious that}
 \begin{align}
II &\leq\f{C}{(\rho-r)}\iint_{Q(\rho)}|v|^{2}\phi^{2\beta-1}\wgr{|v-\nabla\Pi_{h}|}  \nonumber\\
 &\leq \f{C}{(\rho-r)}\iint_{Q(\rho)}|v|^{3}\phi^{2\beta-1} +\phi^{2\beta-1} |v|^{2}|\nabla\Pi_{h}|     \nonumber\\
 &  \wgr{= \f{C}{(\rho-r)}\iint_{Q(\rho)}|v|^{3}\phi^{\beta(3-\alpha)}+\f{C}{(\rho-r)}\iint_{Q(\rho)}\phi^{2\beta-1} |v|^{2}|\nabla\Pi_{h}|}\nonumber\\
 &\wgr{=:}II_{1}+II_{2},\label{13.10}
 \end{align}
 where the fact $2\beta-1-\beta(3-\alpha)=0$ is used.

  Utilizing \wgr{similar argument as in proof of} \eqref{zw} in Lemma \ref{zc} and the Young inequality, we write
 $$\ba
 II_{1} &\leq \f{\wgr{C}}{(\rho-r)}\rho^{3(\alpha-1)/2} \|v\|_{L^{p,q}(Q(\rho))}^{\alpha}\Big(\|v\phi^{\beta}\|_{\wgr{L^{\infty,2}}(Q(\rho))}^{2}+\|\wgr{\nabla(v\phi^{\beta})}\|_{L^{2}(Q(\rho))}^{2}\Big)^{(3-\alpha)/2}\\
&\leq  \f{C}{(\rho-r)^{\f{2}{\alpha-1}}}\rho^{3 } \|u\|_{L^{p,q}(Q(\rho))}^{\f{2\alpha}{\alpha-1}}+\f{1}{32}\Big(\|v\phi^{\beta}\|_{\wgr{L^{\infty,2}}(Q(\rho))}^{2}+\|\wgr{\nabla(v\phi^{\beta})}\|_{L^{2}(Q(\rho))}^{2}\Big).
\ea$$
Invoking the  H\"older inequality,
\eqref{2.3.8}, \wgr{\eqref{3.2}} and the Young inequality, we obtain
\begin{align}
II_{2}
 &\leq  \f{C}{(\rho-r)}\|v\phi^{ \beta }\|_{L^{\infty,2}}\|v \|_{L^{2,2}}\|\nabla\Pi_{h} \|_{L^{2,\infty}}\nonumber\\
 &\leq \f{1}{32}\|v\phi^{ \beta }\|^{2}_{L^{\infty,2}}+ \f{C}{(\rho-r)^{2}}\|v \|^{2}_{L^{2,2}}\f{C}{(\rho-r)^{3}}\| \nabla\Pi_{h} \|^{2}_{L^{2,2}(Q( \rho ))}\nonumber\\
 &\leq \f{1}{32}\|v\phi^{ \beta }\|^{2}_{L^{\infty,2}}+
 \f{C\rho^{\f{10\alpha-8}{ \alpha}} }{(\rho-r)^{5}}\| u \|^{4}_{L^{p,q}(Q( \rho ))}.\label{13.11}
\end{align}
Arguing in the same manner as in  \wgr{\eqref{13.10}}, we infer that
 \begin{align}
III&=
\iint_{Q(\rho)} \phi^{2\beta}(\wgr{(v-\nabla\Pi_{h})}\otimes v :\nabla^{2}\Pi_{h} )  \nonumber\\
\leq&
\iint_{Q(\rho)} \phi^{2\beta}|v|^{2}|\nabla^{2}\Pi_{h}|+\iint_{Q(\rho)}\phi^{2\beta}|v||\nabla\Pi_{h}| |\nabla^{2}\Pi_{h}|
\nonumber\\
\wgr{=:}&III_{1}+III_{2}\nonumber \end{align}
From the H\"older inequality, Young's inequality and
\eqref{2.3.9}, \wgr{\eqref{3.2}}, we see that
\begin{align}
III_{1}\leq & C\|v\phi^{ \beta }\|_{L^{\infty,2}}\|v \|_{L^{2,2}}\|\nabla^{2}\Pi_{h} \|_{L^{2,\infty}}  \nonumber
\\
\leq & \f{1}{32}\|v\phi^{ \beta }\|^{2}_{L^{\infty,2}(Q(\f{r+\rho}{2}))}+ C\|v \|^{2}_{L^{2,2(Q(\f{r+\rho}{2}))}}\|\nabla^{2}\Pi_{h} \|^{2}_{L^{2,\infty}(Q(\f{r+\rho}{2}))} \nonumber
\\
\leq & \f{1}{32}\|v\phi^{ \beta }\|^{2}_{L^{\infty,2}}+ C\wgr{\|v\|^{2}_{L^{2,2}}}\f{C}{(\rho-r)^{5}}\| \nabla\Pi_{h} \|^{2}_{L^{2,2}(Q( \rho ))} \nonumber
\\
 \leq & \f{1}{32}\|v\phi^{ \beta }\|^{2}_{L^{\infty,2}}+ \f{C\rho^{\f{10\alpha-8}{ \alpha}} }{(\rho-r)^{5}}\| u \|^{4}_{L^{p,q}(Q( \rho ))}\wgr{.}
\end{align}
Similarly, we have
\begin{align}
III_{2}\leq &  C\|v\phi^{ \beta }\|_{L^{\infty,2}}\|\nabla\Pi_{h} \|_{L^{2,2}}\|\nabla^{2}\Pi_{h} \|_{L^{2,\infty}} \nonumber
\\
\leq & \f{1}{32}\|v\phi^{ \beta }\|^{2}_{L^{\infty,2}(Q(\f{r+\rho}{2}))} +C\|\nabla\Pi_{h} \|^{2}_{L^{2,2}(Q(\f{r+\rho}{2}))}\|\nabla^{2}\Pi_{h} \|^{2}_{L^{2,\infty}(Q(\f{r+\rho}{2}))}\nonumber
\\
\leq & \f{1}{32}\|v\phi^{ \beta }\|^{2}_{L^{\infty,2}} +C\|\nabla\Pi_{h} \|^{2}_{L^{2,2}}\f{C}{(\rho-r)^{5}}\| \nabla\Pi_{h} \|^{2}_{L^{2,2}(Q( \rho ))}\nonumber
\\
\leq& \f{1}{32}\|v\phi^{ \beta }\|^{2}_{L^{\infty,2}}+ \f{C\rho^{\f{10\alpha-8}{ \alpha}} }{(\rho-r)^{5}}\| u \|^{4}_{L^{p,q}(Q( \rho ))}.
\end{align}
As a consequence, we have
\begin{align}
III
 \leq \wgr{\f{2}{32}}\|v\phi^{ \beta }\|^{2}_{L^{\infty,2}}+ \f{C\rho^{\f{10\alpha-8}{ \alpha}} }{(\rho-r)^{5}}\| u \|^{4}_{L^{p,q}(Q( \rho ))}\wgr{.}
\end{align}
In light  of H\"older's inequality, \eqref{wp2}, \wgr{\eqref{3.2}} and Young's inequality, we deduce that
\begin{align}
IV&\leq \f{C}{(\rho-r)}\| v\|_{L^{2}(Q(\f{r+\rho}{2}))}
\| \Pi_{1} \|_{L^{2}(Q(\f{r+\rho}{2}))} \nonumber\\
&\leq \f{C}{(\rho-r)^{2}}\| v\|^{2}_{L^{2}(Q(\f{r+\rho}{2}))}
+\f{1}{16}\| \Pi_{1} \|^{2}_{L^{2}(Q(\rho))} \nonumber\\
&\leq\f{C}{(\rho-r)^{2}}\wgr{\|v\|^{2}_{L^{2}(Q(\f{r+\rho}{2}))}}+\f{1}{16}\| \nabla u\|^{2}_{L^{2}(Q(\rho))}\nonumber\\
&\leq\f{C\rho^{\f{5\alpha-4}{ \alpha}}}{(\rho-r)^{2}} \| u \|^{2}_{L^{p,q}(Q( \rho ))}+\f{1}{16}\| \nabla u\|^{2}_{L^{2}(Q(\rho))}.
\end{align}

To proceed further, we denote $\eta=\phi^{\beta-1}$.
The fact  $\partial_{i}\partial_{i} \wgr{\Pi_2}=-\partial_{i}\partial_{j} (u_{i}u_{j})$ and Leibniz's formula allow us to get
			\[\partial_{i}\partial_{i}(\wgr{\Pi_2}\eta)=-\eta \partial_{i}\partial_{j}(u_{j}u_{i})+
			2\partial_{i}\eta\partial_{i}\Pi+\Pi\partial_{i}\partial_{i}\eta.\]
Integrating by parts, we have
			\be\ba\label{p}
			\eta\Pi_{2}(x)=&\Gamma \ast (-\eta \partial_{i}\partial_{j}(u_{j}u_{i})+			2\partial_{i}\eta\partial_{i}\Pi_{2}+\Pi_{2}\partial_{i}\partial_{i}\eta)\\
			=&-\partial_{i}\partial_{j}\Gamma \ast (\eta (u_{j}u_{i}))\\
			&+2\partial_{i}\Gamma \ast(\partial_{j}\eta(u_{j}u_{i}))\wgr{+2}\partial_{i}\Gamma \ast(\partial_{i}\eta \Pi_{2}) \\
			&-\Gamma \ast
			(\partial_{i}\partial_{j}\eta u_{j}u_{i})-\Gamma \ast(\partial_{i}\partial_{i}\eta \Pi_{2})\\
			\wgr{=:} & \Pi_{21}(x)+\Pi_{22}(x)+\Pi_{23}(x),
			\ea\ee
By   Young's convolution inequality, setting $\tau=\f{5q-12}{2q}>0$, we arrive at
\begin{align}
\|\Pi_{22}(x)\|_{L^{1,2}}&\leq \f{C\rho^{\tau}}{\rho-r} \B(\|\Pi_{2}\|_{L^{1,\f{6}{5-2\tau}}}+ \||u|^{2}\|_{L^{1,\f{6}{5-2\tau}}}\B)\nonumber\\&\leq \f{C\rho^{\tau}}{\rho-r}\|u \|^{2}_{L^{2,q}}\nonumber\\
&\leq\f{C\rho^{\tau+\f{2(p-2)}{p}}}{\rho-r}\|u \|^{2}_{L^{p,q}}\nonumber\\
&\leq\f{C\rho^{ \f{9\alpha-8}{2\alpha}}}{\rho-r}\|u \|^{2}_{L^{p,q}},\label{13.17}\end{align}
Likewise,
\begin{align}
 \|\Pi_{23}(x)\|_{L^{1,2}}&\leq \f{C\rho}{(\rho-r)^{2}} \B(\|\Pi_{2}\|_{L^{1,\f{6}{5}}}+ \||u|^{2}\|_{L^{1,\f{6}{5}}}\B)\nonumber\\ &\leq \f{C\rho}{(\rho-r)^{2}}\|u \|^{2}_{L^{2,\f{12}{5}}}\nonumber\\&\leq \f{C\rho^{\f{11\alpha-8}{2\alpha}}}{(\rho-r)^{2}}\|u \|^{2}_{L^{p,q}}.\label{13.18}
\end{align}
Note that $$u_{i}u_{j}= v_{i}v_{j}-(\nabla\Pi_{h})_{i}(v)_{j}+(\nabla\Pi_{h})_{i}(\nabla\Pi_{h})_{j}-(v)_{i}(\nabla\Pi_{h})_{j}\wgr{.}$$
Hence, some integrations by parts ensure that
$$\ba
\wgr{\Pi_{21}}&=-\partial_{i}\partial_{j}\Gamma \ast \B[\eta (v_{i}v_{j}-(\nabla\Pi_{h})_{i}(v)_{j}+(\nabla\Pi_{h})_{i}(\nabla\Pi_{h})_{j}-(v)_{i}(\nabla\Pi_{h})_{j}\B]
\\&=-\partial_{i}\partial_{j}\Gamma \ast \B[\eta (v_{i}v_{j})\B]+2\partial_{i}\partial_{j}\Gamma\wgr{\ast} \B[\eta (\nabla\Pi_{h})_{i}(v)_{j}\B] -\partial_{i}\partial_{j}\Gamma \ast\B[\eta (\nabla\Pi_{h})_{i}(\nabla\Pi_{h})_{j} \B]\\
&\wgr{=:}\Pi_{211}+\Pi_{212}+\Pi_{213}\wgr{.}
\ea$$
	It is clear that		
$$\ba
\wgr{V} &\leq
\f{C}{\rho-r}\iint_{Q(\rho)} |v \phi^{\beta}|  (|\Pi_{211}|+|\Pi_{212}|+|\Pi_{213}|+|\Pi_{22}|+|\Pi_{23}|)\\
&\wgr{=: V_{1}+V_{2}+ V_{3}+V_{4}+V_{5}}.\ea$$
To bound $IV_{1}$,
  the classical Calder\'on-Zygmund theorem allows us to argue as the deduction of  \wgr{\eqref{zw}} to obtain that
\begin{align}
\wgr{V_{1}}&\leq
  \f{C}{(\rho-r)} \|v \phi^{\beta} \|_{L^{m,n}(Q(\rho))}
\|\Pi_{211}\|_{L^{m^{\ast},
n^{\ast}}}
 \nonumber \\
&\leq \f{C}{(\rho-r)} \|v \phi^{\beta} \|_{L^{m,
n}(Q(\rho))}
\|v^{\alpha} (v \phi^{\beta})^{2-\alpha} \phi^{\beta-1-\beta(2-\alpha)} \|_{L^{m^{\ast},
n^{\ast}}}
  \nonumber\\&\leq \f{C}{(\rho-r)} \|v \phi^{\beta} \|^{3-\alpha}_
  {L^{(3-\alpha)(\f{p}{\alpha})^{\ast},(3-\alpha)(\f{q}{\alpha})^{\ast}}(Q(\rho))}
\|v  \|^{\alpha}_{L^{p,q}}\nonumber\\
&\leq \f{\wgr{C}}{(\rho-r)^{\f{2}{\alpha-1}}}\rho^{3 } \|\wgr{u}\|_{L^{p,q}(Q(\rho))}
^{\f{2\alpha}{\alpha-1}}+\f{1}{32}\Big(\|v\phi^{\beta}\|
_{\wgr{L^{\infty,2}}(Q(\rho))}^{2}+\|\wgr{\nabla(v\phi^{\beta})}\|_{L^{2}(Q(\rho))}^{2}\Big)\wgr{,}\label{locp5}
\end{align}
where we used the Young inequality and the fact that $\beta-1-\beta(2-\alpha)=0$.

 Likewise, we get
\begin{align}
\wgr{V_{2}}&\leq\f{C}{\rho-r} \|v\phi^{\beta}\|_{L^{\infty,2}}\|\Pi_{212}\|_{L^{1,2}}\nonumber\\
&\leq \f{C}{\rho-r}\|v\phi^{\beta}\|_{L^{\infty,2}}\| \eta (\nabla\Pi_{h})_{i}(v)_{j}\|_{L^{1,2}}\nonumber\\
&\leq \wgr{\f{C}{\rho-r}}\|v\phi^{\beta}\|_{L^{\infty,2}}\|\nabla\Pi_{h}\|_{L^{2,\infty}}\| \eta (v)_{j}\|_{L^{2,2}}\nonumber\\
&\leq \f{1}{128}\| \phi^{\beta}  v \|^{2}_{L^{\infty,2}}+\f{C}{(\rho-r)^{2}}\|v\|^{2}_{L^{2}}\|\nabla\Pi_{h}\|^{2}_{L^{2,\infty}}
\nonumber\\
&\leq \f{1}{128}\| \phi^{\beta}  v \|^{2}_{L^{\infty,2}}+ \f{C\rho^{\f{10\alpha-8}{ \alpha}} }{(\rho-r)^{5}}\| u \|^{4}_{L^{p,q}(Q( \rho ))}.
\end{align}
\wgr{Similarly}, we get
\begin{align}
\wgr{V_{3}}&\leq \f{\wgr{C}}{\rho-r} \|v\phi^{\beta}\|_{L^{\infty,2}}\|\Pi_{213}\|_{L^{1,2}}\nonumber\\
&\leq \f{\wgr{C}}{\rho-r} \|v\phi^{\beta}\|_{L^{\infty,2}}\|  \nabla\Pi_{h} \|^{2}_{L^{2,4}}\nonumber\\
&\leq\f{1}{128}\| \phi^{\beta}  v \|^{2}_{L^{\infty,2}}+\f{\wgr{C}}{(\rho-r)^{2}} \|  \nabla\Pi_{h} \|^{4}_{L^{2,4}}\nonumber\\
&\leq\f{1}{128}\| \phi^{\beta}  v \|^{2}_{L^{\infty,2}}+\f{C \rho^{\f{13\alpha-8}{\alpha}}}{(\rho-r)^{8}}\|u \|^{4}_{L^{p,q}}.
\end{align}
Combining the H\"older inequality and \eqref{13.17} ensures that
\begin{align}
\wgr{V_{4}}&\leq \f{\wgr{C}}{\rho-r} \|v\phi^{\beta}\|_{L^{\infty,2}}\| \Pi_{22}\|_{L^{1,2}}\nonumber\\
&\leq \|v\phi^{\beta}\|_{L^{\infty,2}}\f{C\rho^{ \f{9\alpha-8}{2\alpha}}}{(\rho-r)^{2}}\|u \|^{2}_{L^{p,q}}\nonumber\\
&\leq\f{1}{128}\| \phi^{\beta}  v \|^{2}_{L^{\infty,2}}+\f{C\rho^{ \f{9\alpha-8}{ \alpha}}}{(\rho-r)^{4}}\|u \|^{4}_{L^{p,q}}.
\end{align}
Arguing as in the proof of the last inequality, we deduce that
$$\ba
\wgr{V_{5}}&\leq \f{C }{(\rho-r)} \|v\phi^{\beta}\|_{L^{\infty,2}}\| \Pi_{23}\|_{L^{1,2}}\\
&\leq \|v\phi^{\beta}\|_{L^{\infty,2}}\f{C\rho^{\f{11\alpha-8}{2\alpha}}}{(\rho-r)^{3}}\|u \|^{2}_{L^{p,q}}\\
&\leq\f{1}{128}\| \phi^{\beta}  v \|^{2}_{L^{\infty,2}}+\f{C\rho^{\f{11\alpha-8}{\alpha}}}{(\rho-r)^{6}}\|u \|^{4}_{L^{p,q}}.
\ea$$
We derive from the Cauchy-Schwarz inequality and \eqref{3.2} that
\begin{align}
 \iint_{Q(\rho)}|\nabla( v \phi^{\beta})|^{2} dxds
\leq&2\B( \iint_{Q(\rho)}|\nabla    v|^{2}\phi^{2\beta} dxds
+\beta^{2} \iint_{Q(\rho)}|\nabla\phi|^{2}|v|^{2}\phi^{2\beta-2}dxds\B) \nonumber\\
\leq&2  \iint_{Q(\rho)}|\nabla    v|^{2}\phi^{2\beta} dxds
+\wgr{\f{C\rho^{\f{5\alpha-4}{ \alpha}} }{(\rho-r)^{2}}\| u \|^{2}_{L^{p,q}(Q( \rho ))}}.\label{cz}
\end{align}
Inserting \wgr{all these estimates} into \eqref{wloc1} and using \eqref{cz}, we conclude  that
\begin{align}
  &\sup_{-\rho^{2}\leq t\leq0}\int_{B(\rho)}|v\phi^{\beta}|^2   d  x+ \iint_{Q(\rho)}\big|\nabla( v\phi^{\beta})\big|^2  d  x d  \tau\nonumber\\\leq& \f{1}{4} \B(\|v\phi^{\beta}\|_{L^{2,\infty}(Q(\rho))}^{2}+\|\nabla (v\phi^{\beta})\|_{L^{2}(Q(\rho))}^{2}\B) +\f{C\rho^{\f{5\alpha-4}{ \alpha}} }{(\rho-r)^{2}}\| u \|^{2}_{L^{p,q}(Q( \rho ))}\nonumber\\
  &+\f{C}{(\rho-r)^{\f{2}{\alpha-1}}}\rho^{3 } \|u\|_{L^{p,q}(Q(\rho))}^{\f{2\alpha}{\alpha-1}}+
 \f{C\rho^{\f{10\alpha-8}{ \alpha}} }{(\rho-r)^{5}}\| u \|^{4}_{L^{p,q}(Q( \rho ))}+\f{1}{16}\| \nabla u\|^{2}_{L^{2}(Q(\rho))}\nonumber\\&
+\wgr{\f{C\rho^\f{13\alpha-8}{ \alpha} }{(\rho-r)^{8}}}\| u \|^{4}_{L^{p,q}(Q( \rho ))}
 +\f{C\rho^{ \f{9\alpha-8}{ \alpha}}}{(\rho-r)^{4}}\|u \|^{4}_{L^{p,q}}+\f{C\rho^{\f{11\alpha-8}{\alpha}}}{(\rho-r)^{6}}\|u \|^{4}_{L^{p,q}}\wgr{.}\nonumber
  \end{align}
This in turn implies
  \begin{align}
  &\sup_{-\rho^{2}\leq t\leq0}\int_{B(\rho)}|v\wgr{\phi^{\beta}}|^2   d  x+ \iint_{Q(\rho)}\big|\nabla( v\wgr{\phi^{\beta}})\big|^2  d  x d  \tau\nonumber\\\leq&     \f{C\rho^{\f{5\alpha-4}{ \alpha}} }{(\rho-r)^{2}}\| u \|^{2}_{L^{p,q}(Q( \rho ))}+
\f{C}{(\rho-r)^{\f{2}{\alpha-1}}}\rho^{3 } \|u\|_{L^{p,q}(Q(\rho))}^{\f{2\alpha}{\alpha-1}}\nonumber\\&
 +
 \f{C\rho^{\f{10\alpha-8}{ \alpha}} }{(\rho-r)^{5}}\| u \|^{4}_{L^{p,q}(Q( \rho ))}+\f{1}{16}\| \nabla u\|^{2}_{L^{2}(Q(\rho))}\nonumber\\&
+\wgr{\f{C\rho^{\f{13\alpha-8}{ \alpha}} }{(\rho-r)^{8}}}\| u \|^{4}_{L^{p,q}(Q( \rho ))}
 +\f{C\rho^{ \f{9\alpha-8}{ \alpha}}}{(\rho-r)^{4}}\|u \|^{4}_{L^{p,q}}+\f{C\rho^{\f{11\alpha-8}{\alpha}}}{(\rho-r)^{6}}\|u \|^{4}_{L^{p,q}}\wgr{.}\label{keyl}
  \end{align}
\wgr{The interior estimate} of harmonic function
  \eqref{h1} and \eqref{wp1} implies that
$$\ba\|\nabla\Pi_{h}\|^{2}_{L^{2,6}Q(r)}&\leq \f{Cr^{\f{1}{6}\cdot3\cdot2}}{(\rho-r)^{2\cdot3\cdot\f{1}{2}}}
\|\nabla\Pi_{h}\|^{2}_{L^{2}Q(\rho)}\\
&\leq\f{\wgr{C}\rho^{\f{6\alpha-4}{\alpha}}}{(\rho-r)^{ 3}}\|u\|^{2}_{L^{p,q}
 (Q(\rho))}.
\ea$$
With the help of the triangle inequality, interpolation inequality  \eqref{sampleinterplation} and the last inequality, we get
 \begin{align}
 \|u\|^{2}_{L^{2,6}(Q(r))} \leq& \|v\|^{2}_{L^{2,6}(Q(r))}+\|\nabla\Pi_{h}\|^{2}_{L^{2,6}(Q(r))}\nonumber \\
\leq& C\B\{\|v\|_{L^{2,\infty}(Q(r))}^{2}+\|\nabla v\|_{L^{2}(Q(r))}^{2}\B\}+\f{\wgr{C}\rho^{\f{6\alpha-4}{\alpha}}}{(\rho-r)^{ 3}}\|u\|^{2}_{L^{p,q}
 (Q(\rho))} \label{l2l6}.
\end{align}
Employing \eqref{h1} and \wgr{\eqref{2.3.8}} once again, we have the estimate
$$
 \|\nabla^{2} \Pi_{h}\|^{2}_{L^{2}(Q(r))}\leq \f{Cr^{3}}{(\rho-r)^{3+2\cdot1}} \|\nabla\Pi_{h}\|^{2}_{L^{2}(Q(\f{r+\rho}{2}))}\leq \f{\wgr{C}\rho^{\f{8\alpha-4}{\alpha}}}{\wgr{(\rho-r)^{5}}}\|u\|^{2}_{L^{p,q}
 (Q(\rho))}.
$$
This together with the triangle inequality and \eqref{keyl} leads to
\begin{align}
  \|\nabla u\|^{2}_{L^{2}(Q(r))}\leq & \|\nabla v\|^{2}_{L^{2}(Q(r))}+
 \|\nabla^{2} \Pi_{h}\|^{2}_{L^{2}(Q(r))}\nonumber\\
 \leq &  \B\{\wgr{1+\f{\rho^{3}}{(\rho-r)^{3}}}\B\}\f{C\rho^{\f{5\alpha-4}{ \alpha}} }{(\rho-r)^{2}}\| u \|^{2}_{L^{p,q}(Q( \rho ))}+
 \f{C}{(\rho-r)^{\f{2}{\alpha-1}}}\rho^{3 } \|u\|_{L^{p,q}(Q(\rho))}^{\f{2\alpha}{\alpha-1}}\nonumber\\&
 +
 \f{C\rho^{\f{10\alpha-8}{ \alpha}} }{(\rho-r)^{5}}\| u \|^{4}_{L^{p,q}(Q( \rho ))}+\f{1}{16}\| \nabla u\|^{2}_{L^{2}(Q(\rho))}\nonumber\\&
+\wgr{\f{C\rho^{\f{13\alpha-8}{ \alpha}} }{(\rho-r)^{8}}}\| u \|^{4}_{L^{p,q}(Q( \rho ))}
 +\f{C\rho^{ \f{9\alpha-8}{ \alpha}}}{(\rho-r)^{4}}\|u \|^{4}_{L^{p,q}}+\f{C\rho^{\f{11\alpha-8}{\alpha}}}{(\rho-r)^{6}}\|u \|^{4}_{L^{p,q}}.\label{tidul2}
\end{align}
As an immediate consequence of  \eqref{l2l6} and \eqref{tidul2}, we get
  \begin{align}
&\|u\|^{2}_{L^{2,6}(Q(r))}+ \|\nabla u\|^{2}_{L^{2}(Q(r))}
\nonumber\\\leq&
    \B\{1+\f{\rho }{(\rho-r) }+\f{\rho^{3}}{(\rho-r)^{3}}\B\}\f{C\rho^{\f{5\alpha-4}{ \alpha}} }{(\rho-r)^{2}}\| u \|^{2}_{L^{p,q}(Q( \rho ))}+
 \f{C}{(\rho-r)^{\f{2}{\alpha-1}}}\rho^{3 } \|u\|_{L^{p,q}(Q(\rho))}^{\f{2\alpha}{\alpha-1}}\nonumber\\&
 +
 \f{C\rho^{\f{10\alpha-8}{ \alpha}} }{(\rho-r)^{5}}\| u \|^{4}_{L^{p,q}(Q( \rho ))}+\f{1}{16}\| \nabla u\|^{2}_{L^{2}(Q(\rho))}\nonumber\\&
+\wgr{\f{C\rho^{\f{13\alpha-8}{ \alpha}} }{(\rho-r)^{8}}}\| u \|^{4}_{L^{p,q}(Q( \rho ))}
 +\f{C\rho^{ \f{9\alpha-8}{ \alpha}}}{(\rho-r)^{4}}\|u \|^{4}_{L^{p,q}}+\f{C\rho^{\f{11\alpha-8}{\alpha}}}{(\rho-r)^{6}}\|u \|^{4}_{L^{p,q}}.\nonumber
\end{align}
Now, we are in a position to apply     \cite[Lemma V.3.1,   p.161]{[Giaquinta]} to the latter to find that
\begin{align}
\|u\|^{2}_{L^{2,6}(Q(\f{R}{2}))}&+ \|\nabla u\|^{2}_{L^{2}(Q(\f{R}{2}))}\nonumber\\
\leq
  CR^{\f{3\alpha-4}{\alpha}}\|  u\|^{2}_{L^{p,q}(Q(R))}
  +&CR^{ \f{5\alpha-8}{\alpha}}\|  u\|^{4}_{L^{ p,q} (Q(R))}
  +CR^{\f{3\alpha-5}{\alpha-1}}\|  u\|^{\f{2\alpha}{\alpha-1}}_{L^{p,q}(Q(R))}.
  \end{align}
  This achieves the proof of this  proposition.
\end{proof}

\section{Induction arguments and proof of Theorem \ref{the1.2}}
\label{sec4}
In this section, we utilize \wgr{an especial} case of \eqref{wwzc} with $p=q=(5+2\delta)/2$ and induction arguments to
finish the proof of Theorem \ref{the1.2}. To this end, we
begin with a  critical proposition, which can be seen   as the bridge between the previous step and the next step for the given statement in the
induction arguments.
\begin{prop}\label{keyinindu}
Assume that $\iint_{\tilde{Q}(r)}  |v|^{\f{10}{3}}\leq r^{5}N,$ with
$r_{k}\leq r\leq r_{k_{0}}$.
There is a constant $C$ such that the following result holds. For any given $(x_{0},\,t_{0})\in\mathbb{R}^{n}\times \mathbb{R}^{-}$ and $k_{0}\in\mathbb{N}$, we have
for any $k>k_{0}$,
 \begin{align}
&\sup_{-r_{k}^{2}\leq t-t_0\leq 0}\fbxo |v|^{2}
+r_{k}^{-3}\iint_{\tilde{Q} _{k}}
 |\nabla v |^{2}\nonumber\\
\leq&  C\sup_{-r_{k_{0}}^{2}\leq t-t_0\leq 0}\fbxozero|v|^{2}+C\sum^{k}_{l=k_{0}} r_{l}\Big(\fqxol |v|^{\f{10}{3}}   \Big)^{\f{9}{10}}\nonumber \\
&
+C\sum^{k}_{l=k_{0}}r^{ \f{1+2\delta}{5+2\delta}}_{l} \B(\fqxol   |v|^{\f{10}{3} } \Big)^{\f{3}{5}}
\Big(\iint_{\wgr{{Q}_{1}}} |u|^{\f{5+2\delta}{2}}   \Big)^{\f{2}{5+2\delta}}   \nonumber\\&+C\sum^{k}_{l=k_{0}}r^{ \f{6+4\delta}{5+2\delta}}_{l} \B(\fqxol  |v|^{\f{10}{3} } \Big)^{\f{3}{5}}
\Big(\iint_{\wgr{{Q}_{1}}} |u|^{\f{5+2\delta}{2}}   \Big)^{\f{2}{5+2\delta}} \\&+C\sum^{k}_{l=k_{0}}r_{l}^{\f{2+4\delta}{\wgr{5+2\delta}}}r_{l}^{-\f{3}{2}}\|v\|_{L^{\f{5+2\delta}
{1+2\delta},\f{6(5+2\delta)}{11-2\delta}}(\tilde{Q}_{l})} \|u\|^{2}_{
L^{\f{5+2\delta}{2}} (\wgr{{Q}(1)})} \nonumber\\&
+C\sum^{k }_{l=k_{0}} r_{l}\B(\fqxol  |v|^{\f{10}{3} } \Big)^{\f{3}{10}}
\Big(\iint_{\tilde{Q}_{\wgr{k_0}} } |\nabla u|^{2 }   \Big)^{\f{1}{2}}
 \nonumber\\&+
 C
\sum^{k}_{l=k_{0}}r_{l}^{\wgr{\f{1+2\delta}{5+2\delta}}}
\Big(\fqxol |v|^{\f{10}{3} }   \Big)^{\f{3}{10}}\B\{
 N^{3/5}+N^{\f{3}{10}} \|u\|_{L^{\f{5+2\delta}{2}}(\wgr{{Q}(1)})}\B\}\nonumber\\
  &+C\sum^{k}_{l=k_{0}}r_{l}^{\f{\wgr{1+2\delta}}{5+2\delta}}r_{l}^{-\f{3}{2}}
  \|v\|_{L^{\f{5+2\delta}{1+2\delta},
  \f{6(5+2\delta)}{11-2\delta}}(\tilde{Q}_{l} )} \|u\|^{2}_{
L^{\f{5+2\delta}{2}} (\wgr{{Q}(1)})}.\label{eq4.1}
 \end{align}
\end{prop}
\begin{proof}
In order to simplify the presentation, we suppose $(x_{0},t_{0})=(0,0)$.
Let us introduce the backward heat kernel
$$
\Gamma(x,t)=\frac{1}{4\pi(r_{k}^2-t)^{3/2}}e^{-\frac{|x|^2}
{4(r_{k}^2-t)}}.
$$
Furthermore, we choose the smooth cut-off \wgr{function} below
$$\phi (x,t)=\left\{\ba
&1,\,~~(x,t)\in Q(r_{k_{0}+1}),\\
&0,\,~~(x,t)\in Q^{c}(\f{3}{2}r_{k_{0}+1});
\ea\right.
$$
 satisfying
$$
\wgr{0\leq \phi \leq1}~~\text{and}~~~~r^{2}_{k_{0}}|\partial_{t}\phi  (x,t)|
+r^{l}_{k_{0}}|\partial^{l}_{x}\phi (x,t)|\leq C.
$$
Easy calculations lead to the following \wgr{fact:}
\begin{enumerate}[(i)]
\item There is a constant $c>0$ independent of $r_{k}$ such that, for any $(x,t)\in Q(r_{k})$,
$$
 \Gamma(x,t)\geq c r_{k}^{-3}.
$$
\item For any $(x,t)\in Q(\wgr{r_{k}})$, we have
$$
|\Gamma(x,t)\phi(x,t)| \leq C r_{k}^{-3},~~~~~|\nabla\phi(x,t)\Gamma(x,t)| \leq C r_{k}^{-4}, ~~~~~|\phi(x,t) \nabla\Gamma(x,t)|\leq C r_{k}^{-4}. $$
\item For any $(x,t)\in Q(3r_{k_{0}}/4)\backslash Q(r_{k_{0}}/2)$, one can deduce that
     $$\Gamma(x,t)\leq Cr_{k_{0}}^{-3},~\partial_{i}\Gamma(x,t)\leq Cr_{k_{0}}^{-4},$$ from which it follows that
$$
|\Gamma(x,t)\partial_{t}\wgr{\phi}(x,t)|+|\Gamma(x,t) \Delta\wgr{\phi}(x,t)|+|\nabla\wgr{\phi}(x,t)\nabla\Gamma(x,t)| \leq Cr_{k_{0}}^{-5}.
$$
\item \label{property2} For any $(x,t)\in Q_{l} \backslash Q_{l+1} $,
    $$
    \Gamma\leq  C r_{l+1}^{-3},~\nabla\Gamma \leq C  r_{l+1}^{-4}.
    $$
\end{enumerate}

Now, plugging  $\varphi_{1}=\phi \Gamma$  into the  local energy inequality \eqref{wloc1} and utilizing the fact that $\Gamma_{t}+\Delta\Gamma=0$, we arrive at that
 \begin{align}
  &\int_{B_{1}}|v|^2\phi (x,t)\Gamma  + \int^{t}_{-r^{2}_{k_{0}}}\int_{B_{1}}\big|\nabla v\big|^2\phi (x,s)\Gamma  \nonumber\\  \leq&   \int^{t}_{-r^{2}_{k_{0}}}\int_{B_{1}} | v|^2( \Gamma\Delta \phi +\Gamma\partial_{t}\phi
  +2\nabla\Gamma\nabla\phi )   \nonumber\\&+\int^{t}_{-r^{2}_{k_{0}}}\int_{B_{1}}|v|^{2}v\cdot\nabla(\phi  \Gamma)-|v|^{2}\nabla\Pi_{h} \cdot\nabla\wgr{(\phi\Gamma)} \nonumber\\
& +\int^{t}_{-r^{2}_{k_{0}}}\int_{B_{1}}\Gamma\phi ( v\otimes v-\wgr{\nabla\Pi_{h}\otimes v}:\nabla^{2}\Pi_{h} ) \nonumber\\&  +\int^{t}_{-r^{2}_{k_{0}}}\int_{B_{1}} \Pi_{1}v\cdot\nabla(\Gamma\phi )  +\int^{t}_{-r^{2}_{k_{0}}}\int_{B_{1}}  \Pi_{2}v\cdot\nabla(\Gamma\phi ),  \label{loc21}
\end{align}
where
$$
\nabla\Pi_{1}=\mathcal{W}_{2,B_{1}}(\Delta u),~~~\nabla\Pi_{2}=-\mathcal{W}_{\f{5+2\delta}{2},B_{1}}(\nabla\cdot(u\otimes u) ).
$$
First, we give the low bound estimates of the terms on the left hand side of   inequality \eqref{loc21}.
Indeed, \wgr{by virtue  of (i)}, we know that
 $$\int_{\wgr{B_{1}}} |v|^{2} \phi  \Gamma\geq \wgr{c}{\fbx} |v|^{2},$$
  and
$$
\int^{t}_{-r^{2}_{k_{0}}}\int_{B_{1}}\phi  \Gamma
 |\nabla v |^{2} \geq \wgr{c}r_{k}^{-3}\iint_{Q _{k}}
 |\nabla v |^{2}.$$
 Second, we focus on the estimation of the right hand side of \eqref{loc21}.
As the support of  $\partial_{t}\phi $
is included in
$\wgr{{Q}}(\f{3r_{k_{0}}}{4})/\wgr{{Q}}(\f{r_{k_{0}}}{2}),$
 we deduce
\be\int^{t}_{-r^{2}_{k_{0}}}\int_{B_1}
 |v|^{2}
\Big|\Gamma\Delta \phi +\Gamma\partial_{t}\phi +2\nabla\Gamma\nabla\phi \Big|
\leq C\sup_{-r_{k_{0}}^{2}\leq t\leq0}\fbxozeroo |v|^{2} .
\ee
 The H\"older  inequality and \eqref{property2} entails that
\begin{align}
&\iint_{Q_{k_{0}}}|v|^{2}v\cdot\nabla(\phi  \Gamma) d  \tau \nonumber\\
\leq&\sum^{k-1}_{l=k_{0}}\iint_{Q_{l}/Q_{l+1}}
  |v|^{3}  |\nabla(\phi  \Gamma)|  +\iint_{Q_{k}}  |v|^{3} |\nabla(\phi  \Gamma)|   \nonumber\\
\leq&\wgr{C}\sum^{k}_{l=k_{0}}r_{l}^{-4}\iint_{Q_{l}}
   |v|^{3}  \nonumber\\
 \leq& C\sum^{k}_{l=k_{0}}r_{l}^{-4}\Big(\iint_{Q_{l} } |v|^{\f{10}{3}}   \Big)^{\f{9}{10}}r^{^{\f{1}{2}}}_{l}\nonumber\\
 \leq& C\sum^{k}_{l=k_{0}} r_{l}\Big(\fqxolo |v|^{\f{10}{3}}   \Big)^{\f{9}{10}}\nonumber.
\end{align}
Similar arguments lead to
\begin{align}
&\iint_{Q_{k_{0}}}|v|^{2}\nabla\Pi_{h} \cdot\nabla(\phi\Gamma)  \nonumber\\
\leq&\sum^{k-1}_{l=k_{0}}\iint_{Q_{l}/Q_{l+1}}
  |v|^{2}| \nabla\Pi_{h}| |\nabla(\phi  \Gamma)|  +\iint_{Q_{k}} |v|^{2}| \nabla\Pi_{h}|  |\nabla(\phi  \Gamma)|   \nonumber\\
\leq&\wgr{C}\sum^{k}_{l=k_{0}}r_{l}^{-4}\iint_{Q_{l}}
 |v|^{2}| \nabla\Pi_{h}|  \nonumber\\
 \leq& C\sum^{k}_{l=k_{0}}r_{l}^{-4}\Big(\iint_{Q_{l} } |v|^{\f{10}{3} }   \Big)^{\f{3}{5}}
\Big(\iint_{Q_{l} } |\nabla\Pi_{h}|^{\f{5+2\delta}{2}}   \Big)^{\f{2}{5+2\delta}}r^{\f{4\delta}{5+2\delta}}_{l}\nonumber\\
\leq& C\sum^{k}_{l=k_{0}}r^{-4+3+ \f{4\delta}{5+2\delta} }_{l} \B(\fqxolo  |v|^{\f{10}{3} } \Big)^{\f{3}{5}}r^{3\cdot\f{2}{5+2\delta}}_{l}
\Big(\iint_{Q_{1} } |\nabla\Pi_{h}|^{\f{5+2\delta}{2}}   \Big)^{\f{2}{5+2\delta}}\nonumber\\
\leq& C\sum^{k}_{l=k_{0}}r^{ \f{1+2\delta}{5+2\delta}}_{l} \B(\fqxolo  |v|^{\f{10}{3} } \Big)^{\f{3}{5}}
\Big(\iint_{Q_{1} } |u|^{\f{5+2\delta}{2}}   \Big)^{\f{2}{5+2\delta}}.\label{3ref}
\end{align}
Similar, by H\"older's inequality and \eqref{h1}, we get
\begin{align}
&\iint_{Q_{k_{0}}}|v|^{2}|\nabla^{2} \Pi_{h}| (\phi\Gamma) \nonumber\\
\leq&\sum^{k-1}_{l=k_{0}}\iint_{Q_{l}/Q_{l+1}}
  |v|^{2}| \wgr{\nabla^2\Pi_{h}}| |\wgr{(\phi  \Gamma)}|  +\iint_{Q_{k}} |v|^{2}| \nabla^{2} \Pi_{h}|  | (\phi  \Gamma)|   \nonumber\\
\leq&\wgr{C}\sum^{k}_{l=k_{0}}r_{l}^{-3}\iint_{Q_{l}}
 |v|^{2}| \nabla^{2} \Pi_{h}|  \nonumber\\
 \leq& C\sum^{k}_{l=k_{0}}r_{l}^{-3}\Big(\iint_{Q_{l} } |v|^{\f{10}{3} }   \Big)^{\f{3}{5}}
\Big(\iint_{Q_{l} } |\nabla^{2}\Pi_{h}|^{\f{5+2\delta}{2}}   \Big)^{\f{2}{5+2\delta}}r^{\f{4\delta}{5+2\delta}}_{l}\nonumber\\
\leq& C\sum^{k}_{l=k_{0}}r^{-3+3+ \f{4\delta}{5+2\delta} }_{l} \B(\fqxolo  |v|^{\f{10}{3} } \Big)^{\f{3}{5}}r^{3\cdot\f{2}{5+2\delta}}_{l}
\Big(\iint_{Q_{1} } |\nabla\Pi_{h}|^{\f{5+2\delta}{2}}   \Big)^{\f{2}{5+2\delta}}\nonumber\\
\leq& C\sum^{k}_{l=k_{0}}r^{ \f{6+4\delta}{5+2\delta}}_{l} \B(\fqxolo  |v|^{\f{10}{3} } \Big)^{\f{3}{5}}
\Big(\iint_{Q_{1} } |u|^{\f{5+2\delta}{2}}   \Big)^{\f{2}{5+2\delta}}\label{343f}.
\end{align}
From the H\"older  inequality, \eqref{property2}, \eqref{h1} and \wgr{\eqref{ph}}, we deduce that
\begin{align}
 &\iint_{Q_{k_{0}}}\phi  \Gamma   |v| |\nabla\Pi_{h}||\nabla^{2}\Pi_{h}|\nonumber\\
\leq&  C\sum^{k}_{l=k_{0}}r_{l}^{-3}\|v\|_{L^{\f{5+2\delta}{1+2\delta},\f{6(5+2\delta)}{11-2\delta}}(Q_{l})}\|\nabla\Pi_{h}\|_{
L^{\f{5+2\delta}{2},\f{12(5+2\delta)}{19+14\delta}}(Q_{l})}\|\nabla^{2}\Pi_{h}\|_{
L^{\f{5+2\delta}{2},\f{12(5+2\delta)}{19+14\delta}}(Q_{l})}\nonumber\\
\leq& C\sum^{k}_{l=k_{0}}r_{l}^{-3}\|v\|_{L^{\f{5+2\delta}{1+2\delta},\f{6(5+2\delta)}{11-2\delta}}(Q_{l})}r_{l}^{6\cdot\f{19+14\delta}{12(5+2\delta)}}\|\nabla\Pi_{h}\|_{
\wgr{L^{\f{5+2\delta}{2}}}(Q(1))}\|\wgr{\nabla\Pi_{h}}\|_{
\wgr{L^{\f{5+2\delta}{2}}}(Q(1))}\nonumber\\
\leq& C\sum^{k}_{l=k_{0}}r_{l}^{\f{2+4\delta}{ (5+2\delta)}}r_{l}^{-\f{3}{2}}\|v\|_{L^{\f{5+2\delta}{1+2\delta},\f{6(5+2\delta)}{11-2\delta}}(Q_{l})} \|u\|^{2}_{
L^{\f{5+2\delta}{2}} (Q(1))} \label{3343}.
\end{align}
We introduce  $\chi_{l}=1$ on $|x|\leq7/8 r_{l}$ and $\chi_{l}=0 $ if $ |x|\geq r_{l}$.
\wgr{It is obvious that $\chi_{k_{0}}\phi\Gamma=\phi\Gamma$ on $Q_{k_{0}}$.}
By taking advantage of  the support of $(\chi_{l}-\chi_{l+1})$, we derive from \eqref{property2}  that $|\nabla((\chi_{l}-\chi_{l+1})\phi  \Gamma)|\leq Cr^{-4}_{l+1}$.
\wgr{Applying  (ii)}  yields that $|\nabla( \chi_{k} \phi  \Gamma)|\leq Cr^{-4}_{k} $. Therefore, we write
\begin{align}
\iint_{Q_{k_{0}}}v\cdot\nabla(\phi  \Gamma) \Pi_{1}=&
\sum^{k-1}_{l=k_{0}}\iint_{Q_{l}}v\cdot\nabla((\chi_{l}-\chi_{l+1})\phi  \Gamma) \Pi_{1}+\iint_{Q_{k}}v\cdot\nabla( \chi_{k} \phi  \Gamma) \Pi_{1}\nonumber\\
=&
\sum^{k-1}_{l=k_{0}}\iint_{Q_{l}}
v\cdot\nabla((\chi_{l}-\chi_{l+1})\phi  \Gamma) (\Pi_{1}- \overline{{\Pi_{1}}}_{l})+\iint_{Q_{k }}u\cdot\nabla( \chi_{k} \phi  \Gamma) (\Pi_{1}-\overline{{\Pi_{1}}}_{k})
\nonumber\\
\leq& C
\sum^{k-1}_{l=k_{0}}r^{-4}_{l+1}\iint_{Q_{l}}
|v||\Pi_{1}- \overline{{\Pi_{1}}}_{l}|+r^{-4}_{k}\iint_{Q_{k }}|v||\Pi_{1}-\bar{\Pi_{1}}_{k}|\nonumber\\
=&:I+II.
\end{align}
Combining the  H\"older inequality, \eqref{h2} and \wgr{\eqref{p1}} yield
\begin{align}
I\leq& C
\sum^{k-1}_{l=k_{0}}r^{-4}_{l+1}
\Big(\iint_{Q_{l} } |v|^{\f{10}{3} }   \Big)^{\f{3}{10}}
\Big(\iint_{Q_{l} } |\Pi_{1}- \overline{{\Pi_{1}}}_{l}|^{2 }   \Big)^{\f{1}{2}} r_{l}\nonumber\\
\leq& C \sum^{k-1}_{l=k_{0}}r^{-4+\f32}_{l+1}
 \B(\fqxolo  |v|^{\f{10}{3} } \Big)^{\f{3}{10}} r_{l}^{3\cdot\f{1}{2}+1}
\Big(\iint_{Q_{\wgr{k_0}} } |\Pi_{1}- \overline{{\Pi_{1}}}_{\wgr{k_0}}|^{2 }   \Big)^{\f{1}{2}} r_{l}\nonumber\\
\leq& C
\sum^{k-1}_{l=k_{0}}r_{l+1}
 \B(\fqxolo  |v|^{\f{10}{3} } \Big)^{\f{3}{10}}
\Big(\iint_{Q_{\wgr{k_0}} } |\Pi_{1}|^{2 }   \Big)^{\f{1}{2}}\nonumber\\
\leq& C
\sum^{k-1}_{l=k_{0}}r_{l+1}
 \B(\fqxolo  |v|^{\f{10}{3} } \Big)^{\f{3}{10}}
\Big(\iint_{Q_{\wgr{k_0}} } |\nabla u|^{2 }   \Big)^{\f{1}{2}},\label{2.4.2}
\end{align}
and
\begin{align}
II
\leq& C
 r_{k}
 \B(\fqxolo  |v|^{\f{10}{3} } \Big)^{\f{3}{10}}
\Big(\iint_{Q_{\wgr{k_0}} } |\nabla u|^{2 }   \Big)^{\f{1}{2}}. \label{2.4.3}
\end{align}
It follows from \eqref{2.4.2} and \eqref{2.4.3} that
$$\iint_{Q_{k_{0}}}v\cdot\nabla(\phi  \Gamma) \Pi_{1}\leq C\sum^{k }_{l=k_{0}} r_{l}\B(\fqxolo  |v|^{\f{10}{3} } \Big)^{\f{3}{10}}
\Big(\iint_{Q_{\wgr{k_0}} } |\nabla u|^{2 }   \Big)^{\f{1}{2}}.
$$
To bound the term involving $\Pi_{2}$, noting that
 $$\ba\Delta\Pi_{2}&=-\text{div\,}\text{div\,}(u\otimes u) \\&=-\text{div\,}\text{div\,}(v\otimes v-v\otimes\nabla\Pi_{h}-\nabla\Pi_{h}\otimes v+\nabla\Pi_{h}\otimes \nabla\Pi_{h}),\ea$$
 we decompose $\Pi_{2}$ into two parts
$$\Pi_{2}= \Pi_{21}+\Pi_{22}$$ with
\begin{align}
&\Delta\Pi_{21}=-\text{div\,}\text{div\,}(v\otimes v-v\otimes\nabla\Pi_{h}-\nabla\Pi_{h}\otimes v),\label{2.4.10}\\
 &\Delta\wgr{\Pi_{22}}=-\text{div\,}\text{div\,}(\nabla\Pi_{h}\otimes \nabla\Pi_{h}).\label{2.4.11}
 \end{align}

For $r_{k}\leq r\leq r_{k_{0}}$, we compute directly that
\be\label{4.8}
\iint_{Q(r)}
 |v\otimes v- \overline{ {v\otimes v }}_{\wgr{r}}|^{10/7} \leq \iint_{Q(r)}
 |v|^{20/7}\leq  Cr^{5}\B(\fqxolor  |v|^{\f{10}{3} } \Big)^{6/7}\leq   Cr^{5}N^{6/7}.
\ee
Making use of  H\"older's inequality and \eqref{h1}, we see that
\begin{align}
\|v\otimes\nabla\Pi_{h}- \overline{v\otimes\nabla\Pi_{h}}_{\wgr{r}}\|_{L^{10/7}(Q(r))} &\leq C\|v\|_{L^{\f{10}{3}}(Q(r))}\|\nabla\Pi_{h}\|_{L^{\f{5+2\delta}{2}}(Q(r))}r^{5\cdot\f{4\delta}{5(5+2\delta)}}\nonumber\\
&\leq Cr^{\f{3}{2}+\f{4\delta}{ 5+2\delta }}\B(\fqxolor|v|^{\f{10}{3}}\B)^{\f{3}{10}}r^{\f{6}{5+2\delta}}\|\nabla\Pi_{h}\|_{L^{\f{5+2\delta}{2}}(\wgr{Q(1)})}\nonumber\\
&\leq   Cr^{\f{27+14\delta}{2(5+2\delta)} }N^{\f{3}{10}} \|u\|_{L^{\f{5+2\delta}{2}}(Q(1))}.
\label{4.9}
\end{align}
Since \eqref{4.8} and \eqref{4.9} are valid, one can invoke Lemma \ref{CW} to obtain
\begin{align}
\|\Pi_{21}-\wgr{\overline{\Pi_{21}}_r}\|_{L^{\f{10}{7}}Q(r)}&\leq C
 r^{7/2}\B(\fqxolor  |v|^{\f{10}{3} } \Big)^{3/5}+Cr^{\f{27+14\delta}{2(5+2\delta)} }N^{\f{3}{10}} \|u\|_{L^{\f{5+2\delta}{2}}(Q(1))}\nonumber\\
 &\leq   Cr^{\f{27+14\delta}{2(5+2\delta)} }\B\{
  \B(\fqxolor  |v|^{\f{10}{3} } \Big)^{3/5}+N^{\f{3}{10}} \|u\|_{L^{\f{5+2\delta}{2}}(Q(1))}\B\}\wgr{.}\label{4.14}
\end{align}
For $k\leq l\leq k_{0}$, concatenating the H\"older inequality and \eqref{4.14} yield
\begin{align}
&r_{l}^{-4}\iint_{Q_{l}}|v||\Pi_{21}-\overline{ \Pi_{21}}_{B(r)}|\nonumber\\&\leq
Cr_{l}^{-4}\Big(\iint_{Q_{l} } |v|^{\f{10}{3} }   \Big)^{\f{3}{10}}
\Big(\iint_{Q_{l} } |\Pi_{2}{ -\overline{(\Pi_{2})}_{B_{l}}}|^{\f{10}{7} }   \Big)^{\f{7}{10}}\nonumber\\
&\leq
Cr_{l}^{-4+\f{3}{2}+\f{27+14\delta}{2(5+2\delta)}}\Big(\fqxolo |v|^{\f{10}{3} }   \Big)^{\f{3}{10}}\B\{
  N^{3/5}+N^{\f{3}{10}} \|u\|_{L^{\f{5+2\delta}{2}}(Q(1))}\B\}\nonumber\\
&\leq
Cr_{l}^{\f{1+2\delta}{5+2\delta}}\Big(\fqxolo |v|^{\f{10}{3} }   \Big)^{\f{3}{10}}\B\{
  N^{3/5}+N^{\f{3}{10}} \|u\|_{L^{\f{5+2\delta}{2}}(Q(1))}\B\}.\label{2.4.15}
\end{align}
According  the Poincar\'e inequality for a ball, H\"older's inequality, \eqref{h1} and \wgr{\eqref{ph}}, we get
\begin{align}
&\|\nabla\Pi_{h}\otimes \nabla\Pi_{h}- \overline{\nabla\Pi_{h}\otimes \nabla\Pi_{h}}_{\wgr{r}}\|_{L^{\f{6(5+2\delta)}{19+14\delta}}(B(r))}\nonumber\\
\leq & Cr\|\nabla\Pi_{h}  \nabla^{2}\Pi_{h} \|_{L^{\f{6(5+2\delta)}{19+14\delta}}(B(r))}\nonumber\\
\leq & Cr\|\nabla\Pi_{h}  \|_{L^{\f{12(5+2\delta)}{19+14\delta}}(B(r))}\|  \nabla^{2}\Pi_{h} \|_{L^{\f{12(5+2\delta)}{19+14\delta}}(B(r))}\nonumber\\
\leq & C r^{1+6\cdot\f{19+14\delta}{12(5+2\delta)}}\|\nabla\Pi_{h}  \|_{L^{\f{5+2\delta}{2}}\wgr{(Q(1))}}\|  \nabla \Pi_{h} \|_{L^{\f{5+2\delta}{2}}\wgr{(Q(1))}}\nonumber\\
\leq & Cr^{\f{29+18\delta}{2(5+2\delta)}}\|u  \|^{2}_{L^{\f{5+2\delta}{2}}(B(1))}\nonumber\\
\leq & Cr^{\f{27+14\delta}{2(5+2\delta)}}\|u  \|^{2}_{L^{\f{5+2\delta}{2}}(B(1))},\label{4.15}
\end{align}
which implies that
\begin{align}
\|\nabla\Pi_{h}\otimes \nabla\Pi_{h}- \overline{\nabla\Pi_{h}\otimes \nabla\Pi_{h}}_{\wgr{r}}\|_{L^{\f{5+2\delta}{4},\f{6(5+2\delta)}{19+14\delta}}(Q(r))}&\leq
 Cr^{\f{27+14\delta}{2(5+2\delta)}}\|u  \|^{2}_{L^{\f{5+2\delta}{2}}(Q(1))}.\label{4.16}
\end{align}
In view of \wgr{Lemma} \ref{CW} and \eqref{4.16}, we have
\begin{align}
\| \Pi_{22}-\bar{\Pi}_{22}\|_{L^{\f{5+2\delta}{4},\f{6(5+2\delta)}{19+14\delta}}(Q(r))}
&\leq C
 r^{\f{27+14\delta}{2(5+2\delta)}}\|u  \|^{2}_{L^{\f{5+2\delta}{2}}(B(1))}\wgr{.}\label{4.17}
\end{align}
By combining  H\"older's inequality and \eqref{4.17}, we deduce that
\begin{align}
&r_{l}^{-4}\iint_{Q_{l}}|v||\Pi_{22}-\overline{ \Pi_{22}}_{B_{l}}|\nonumber\\&\leq
r_{l}^{-4}\|v\|_{L^{\f{5+2\delta}{1+2\delta},\f{6(5+2\delta)}{11-2\delta}}(Q_{l})}\|\Pi_{22}-\overline{ \Pi_{22}}_{B_{l}}\|_{
L^{\f{5+2\delta}{4},\f{6(5+2\delta)}{19+14\delta}}(Q_{l})}\nonumber\\
&\leq C\wgr{r_l}^{\f{1+2\delta}{5+2\delta}}r_{l}^{-\f{3}{2}}
\|v\|_{L^{\f{5+2\delta}{1+2\delta},\f{6(5+2\delta)}{11-2\delta}}(Q_{l})} \|u\|^{2}_{
L^{\f{5+2\delta}{2}} (Q(1))}.\label{4.18}
\end{align}
According to  \eqref{4.18} and \eqref{2.4.15}, we know that
\begin{align}
&\iint_{Q_{k_{0}}}v\cdot\nabla(\phi  \Gamma) \Pi_{2}\nonumber\\
\leq& C
\sum^{k-1}_{l=k_{0}}r^{-4}_{l+1}\iint_{Q_{l}}
|v||\Pi_{21}- \overline{{\Pi_{21}}}_{l}|+r^{-4}_{k}\iint_{Q_{k }}|v||\Pi_{21}-\overline{ {\Pi_{21}}}_{k}|\nonumber\\&+C
\sum^{k-1}_{l=k_{0}}r^{-4}_{l+1}\iint_{Q_{l}}
|v||\Pi_{22}- \overline{{\Pi_{22}}}_{l}|+r^{-4}_{k}\iint_{Q_{k }}|v||\Pi_{22}-\overline{ {\Pi_{22}}}_{k}|\nonumber\\
\leq& \wgr{C}
\sum^{k}_{l=k_{0}}r_{l}^{\f{1+2\delta}{5+2\delta}}\Big(\fqxolo |v|^{\f{10}{3} }   \Big)^{\f{3}{10}}\B\{
  \B(\fqxolor  |v|^{\f{10}{3} } \Big)^{3/5}+\B(\fqxolor|v|^{\f{10}{3}}\B)^{\f{3}{10}} \|u\|_{L^{\f{5+2\delta}{2}}(Q(1))}\B\}\nonumber\\
  &+C\sum^{k}_{l=k_{0}}r_{l}^{\f{1+2\delta}{5+2\delta}}
  r_{l}^{-\f{3}{2}}\|v\|_{L^{\f{5+2\delta}{1+2\delta},\f{6(5+2\delta)}{11-2\delta}}(Q(r))} \|u\|^{2}_{
L^{\f{5+2\delta}{2}} (Q(1))}.
\end{align}
Finally, gathering the above inequalities \wgr{yields the desired inequality}.
\end{proof}

With Proposition \ref{keyinindu} at our disposal, we can now give the proof of
  Theorem \ref{the1.2}.
\begin{proof}[Proof of Theorem \ref{the1.2}]
By H\"older's inequality, it suffices to consider the case with $\delta$ sufficient small.
By the  interior estimate \eqref{h1}  of harmonic function   and \eqref{wp1}, we have
 \be\| \nabla \Pi_{h}\|_{L^{\infty}
(\tilde{B}(1/8))}\leq C\|\nabla \Pi_{h}\|_{L^{\f{5+2\delta}{2} }(B(1))}\leq   C\|u\|_{L^{\f{5+2\delta}{2}}(B(1))}.\label{last1}\ee
We temporarily assume that, for any Lebesgue point $(x_{0},t_{0})\in Q(1/8)$,
\be\label{temp}
|v(x_{0},t_{0})|\leq \wgr{C\varepsilon^{\frac{2(1+2\delta)}{3(5+2\delta)}}}.
\ee
It follows from \eqref{last1} and \eqref{temp} that
$$
\| u\|_{L^{\f{5+2\delta}{2},\infty}
(\tilde{Q}(1/8))}\leq \| \nabla \Pi_{h}\|_{L^{\f{5+2\delta}{2},\infty}(\tilde{Q}(1/8))}+
\| v\|_{L^{\f{5+2\delta}{2},\infty}(\tilde{Q}(1/8))}\leq C\|u\|_{L^{\f{5+2\delta}{2} }(Q(1))}\wgr{+C\varepsilon^{\frac{2(1+2\delta)}{3(5+2\delta)}}}.
$$
The well-known Serrin  regularity \wgr{criterion} implies that that $(0, 0)$ is a regular point.
Thus, it is enough to show \eqref{temp} to complete the proof of Theorem  \ref{the1.2}.
In what follows, let   $(x_{0},t_{0})\in Q(1/8)$ and $r_{k}=2^{-k}$.
According to the Lebesgue differentiation theorem, it  suffices to show
\be\label{goal}
 \fqxoo |v|^{\f{10}{3}}+\wgr{r_{k}}^{-\f{3(5+2\delta)}{2(1+2\delta)}}
 \|v\|^{\f{5+2\delta}{1+2\delta}}_{L^{\f{5+2\delta}{1+2\delta},\f{6(5+2\delta)}{11-2\delta}}(\wgr{\tilde{Q}_{k}})}
\leq \wgr{\varepsilon}^{2/3}, ~~k\geq3.\ee
First, we prove that \eqref{goal} holds true for $k=3$.  Indeed, from \eqref{keyl} with $\alpha=\f{5+2\delta}{5}$ in Section  \ref{sec3}, \wgr{Proposition \ref{prop1.2}} and hypothesis \eqref{wwz}, we get
 \begin{align}
  &\sup_{-(\f{3}{8})^{2}\leq t\leq0}\int_{B(\f{3}{8})}|v |^2   d  x+ \iint_{Q(\f{3}{8})}\big|\nabla v\big|^2  d  x d  \tau\nonumber\\\leq&   C \|u\|^{2}_{L^{\f{5+2\delta}{2}}\wgr{(Q(1/2))}}+
   C \|u\|_{L^{\f{5+2\delta}{2}}(Q(1/2))}^{\f{5+2\delta}{\delta}}\nonumber\\&
 + C\|u \|^{4}_{L^{\f{5+2\delta}{2}}\wgr{(Q(1/2))}}+\f{1}{16}\| \nabla u\|^{2}_{L^{2}(Q(1/2))}
  \nonumber\\\leq&     C\|u\|^{2}_{\wgr{L^\f{5+2\delta}{2}(Q(1))}}+C
  \|u\|_{L^{\f{5+2\delta}{2}}\wgr{(Q(1))}}^{\f{5+2\delta}{\delta}} + C\|u \|^{4}_{L^{\f{5+2\delta}{2}}\wgr{(Q(1))}}
  \nonumber\\
  \leq&      C\varepsilon^{\f{4}{5+2\delta}}.\label{2.4.23}
  \end{align}
\wgr{The interpolation inequality \eqref{sampleinterplation} leads} us to obtain
\be\label{eq3.5}\ba
 \B(\iint_{\tilde{Q}(\wgr{\f{1}{8}})}|v |^{\f{10 }{3}}\B)^{\f{3}{10 }}
&\leq C\B(\sup_{-(\f{3}{8})^{2}\leq t<0}\int_{\wgr{{B}}(\f{3}{8})}|v|^{2}\B)^{1/2}+ C\B(\iint_{\wgr{{Q}} (\f{3}{8})} |\nabla  v |^{2}\Big)^{1/2},
\ea\ee
and
$$\|v\|_{L^{\f{5+2\delta}{1+2\delta},\f{6(5+2\delta)}{11-2\delta}}(\tilde{Q}(\wgr{\f{1}{8}}))}\leq C\B(\sup_{-(\f{3}{8})^{2}\leq t<0}\int_{B(\f{3}{8})}|v|^{2}\B)^{1/2}+ C\B(\iint_{\wgr{{Q}} (\f{3}{8})} |\nabla  v |^{2}\Big)^{1/2}.$$
\wgr{These inequalities together with \eqref{2.4.23}} means
$$
   \fqxoth |v|^{\f{10}{3} }+
 (\wgr{\f{1}{8}})^{-\f{3(5+2\delta)}{2(1+2\delta)}}\|v\|^{\f{5+2\delta}{1+2\delta}}_{\wgr{L^{\f{5+2\delta}{1+2\delta},\f{6(5+2\delta)}{11-2\delta}}(\tilde{Q}(\f{1}{8}))}} \leq C\wgr{\varepsilon}^{\f{7}{6}}+C\wgr{\varepsilon^{\f{2}{1+2\delta}}}.$$
The assertion \eqref{goal} with $k=3$ is valid.
Next, we assume that, for any $3\leq l\leq k$,
$$
  \fqxol |v|^{\f{10}{3} }+r_{l}^{-\f{3(5+2\delta)}{2(1+2\delta)}}\|v\|^{\f{5+2\delta}{1+2\delta}}
  _{L^{\f{5+2\delta}{1+2\delta},\f{6(5+2\delta)}{11-2\delta}}(\tilde{Q}_{l})}
\leq \wgr{\varepsilon}^{2/3}.$$
As a consequence,   for any $r_{k}\leq r\leq r_{3}$, we see that
\be
  \fqxolorr |v|^{\f{10}{3} }+r^{-\f{3(5+2\delta)}{2(1+2\delta)}}
  \|v\|^{\f{5+2\delta}{1+2\delta}}
  _{L^{\f{5+2\delta}{1+2\delta},\f{6(5+2\delta)}{11-2\delta}}
  (\tilde{Q}(r))}
\leq C\wgr{\varepsilon}^{2/3}.
\label{2.4.26}\ee
For any $3\leq i\leq k$, we apply Proposition \ref{keyinindu} to $N=C\wgr{\varepsilon}^{2/3}$ \wgr{ and $k_0=3$ to obtain}
\begin{align}
&\sup_{-r_{i}^{2}\leq t-t_0\leq 0}\wgr{\fbxoi} |v|^{2}
+r_{i}^{-3}\iint_{\wgr{\tilde{Q} _{i}}}
 |\nabla v |^{2}\nonumber\\
\leq&  C\sup_{-r_{k_{0}}^{2}\leq t-t_0\leq 0}\fbxozero|v|^{2}+C\sum^{i}_{l=3} r_{l}\Big(\fqxol |v|^{\f{10}{3}}   \Big)^{\f{9}{10}}\nonumber\\
&+C\sum^{i}_{l=3}r^{ \f{1+2\delta}{5+2\delta}}_{l} \B(\fqxol  |v|^{\f{10}{3} } \Big)^{\f{3}{5}}
\Big(\iint_{Q_{1} } |u|^{\f{5+2\delta}{2}}   \Big)^{\f{2}{5+2\delta}}   \nonumber\\
&+C\sum^{i}_{l=3}r^{ \f{6+4\delta}{5+2\delta}}_{l} \B(\fqxol  |v|^{\f{10}{3} } \Big)^{\f{3}{5}}
\Big(\iint_{Q_{1} } |u|^{\f{5+2\delta}{2}}   \Big)^{\f{2}{5+2\delta}} \nonumber\\&+C\sum^{i}_{l=3}r_{l}^{\f{2+4\delta}{ (5+2\delta)}}r_{l}^{-\f{3}{2}}\|v\|_{L^{\f{5+2\delta}{1+2\delta}}
L^{\f{6(5+2\delta)}{11-2\delta}}(Q_{l})} \|u\|^{2}_{
L^{\f{5+2\delta}{2}} (Q(1))} \nonumber\\
&
+C\sum^{i}_{l=3} r_{l}\B(\fqxol  |v|^{\f{10}{3} } \Big)^{\f{3}{10}}
\Big(\iint_{\wgr{\tilde{Q}_{k_0}}} |\nabla u|^{2 }   \Big)^{\f{1}{2}}
 \nonumber\\
 &+
 \wgr{C}\sum^{i}_{l=3}r_{l}^{\f{1+2\delta}{5+2\delta}}\Big(\fqxol |v|^{\f{10}{3} }   \Big)^{\f{3}{10}}\B\{
  \B(\fqxolor  |v|^{\f{10}{3} } \Big)^{3/5}+\B(\fqxolor|v|^{\f{10}{3}}\B)^{\f{3}{10}} \|u\|_{L^{\f{5+2\delta}{2}}(Q(1))}\B\}\nonumber\\
  &+C\sum^{i}_{l=3}r_{l}^{\f{1+2\delta}{5+2\delta}}r_{l}^{-\f{3}{2}}\|v\|_{L^{\f{5+2\delta}{1+2\delta}}
L^{\f{6(5+2\delta)}{11-2\delta}}(\tilde{Q}_{l})} \|u\|^{2}_{
L^{\f{5+2\delta}{2}} (Q(1))}\nonumber
\\\leq&  C\varepsilon^{\f{4}{5+2\delta}}+C\sum^{i}_{l=3} r_{l}\varepsilon^{\f23\cdot\f{9}{10}}
+C\sum^{i}_{l=3}r^{ \f{1+2\delta}{5+2\delta}}_{l} \varepsilon^{\f23\cdot\f{3}{5}+\f{2}{5+2\delta}}
  \nonumber\\&
  +C\sum^{i}_{l=3}r^{ \f{6+4\delta}{5+2\delta}}_{l} \varepsilon^{\f23\cdot\f{3}{5}+\f{2}{5+2\delta}}
+C\sum^{i}_{l=3}r_{l}^{\f{2+4\delta}{ (5+2\delta)}}\varepsilon^{\f23\cdot\f{1+2\delta}{5+2\delta}+\f{4}{5+2\delta}} \nonumber\\&
+C\sum^{i}_{l=3} r_{l}\varepsilon^{\f23\cdot\f{3}{10}}
\varepsilon^{\f{1}{2}\cdot\f{4}{5+2\delta}}
+
 \wgr{C}\sum^{i}_{l=3}r_{l}^{\f{1+2\delta}{5+2\delta}}\varepsilon^{\f23\cdot\f{3}{10}}\B\{
  \varepsilon^{\f23\cdot3/5}+\varepsilon^{\f23\cdot\f{3}{10}+\f{2}{5+2\delta}} \B\} \nonumber\\&+C\sum^{i}_{l=3}r_{l}^{\f{1+2\delta}{5+2\delta}}\varepsilon^{\f23\cdot\f{1+2\delta}{5+2\delta}+\f{4}{5+2\delta}} \nonumber\\\leq& C\varepsilon^{\f{4}{5+2\delta}}+C\sum^{i}_{l=3} r_{l}\varepsilon^{\f35}
+C\sum^{i}_{l=3}r^{ \f{1+2\delta}{5+2\delta}}_{l} \varepsilon^{\f{20+4\delta}{5(5+2\delta)}}
  \nonumber\\&+C\sum^{i}_{l=3}r^{ \f{6+4\delta}{5+2\delta}}_{l} \varepsilon^{\f{20+4\delta}{5(5+2\delta)}}
+C\sum^{i}_{l=3}r_{l}^{\f{2+4\delta}{ (5+2\delta)}}\varepsilon^{\f{14+4\delta}{3(5+2\delta)}}
+C\sum^{i}_{l=3} r_{l}\varepsilon^{\cdot\f{15+2\delta}{5(5+2\delta)}}
 \nonumber\\&+
 \wgr{C}\sum^{i}_{l=3}r_{l}^{\f{1+2\delta}{5+2\delta}}\varepsilon^{\f15}\B\{
  \varepsilon^{\f25}+\varepsilon^{\f{15+2\delta}{5(5+2\delta)}} \B\} +C\sum^{i}_{l=3}r_{l}^{\f{1+2\delta}{5+2\delta}}\varepsilon^{\f{14+4\delta}{3(5+2\delta)}} \nonumber\\
\leq& C \varepsilon^{\f{1 }{2 }},\label{eq3.1}
\end{align}
where the hypothesis \eqref{wwz}
and \eqref{2.4.26} are used.

We derive from  \eqref{sampleinterplation} that
$$\ba
\iint_{\tilde{Q}_{k+1} }|v|^{\f{10}{3}}
\leq& C\B(\sup_{-r^{2 }_{k}\leq t-t_{0}<0}\int_{\tilde{B}_{k}}|v |^{2}\B)^{\f{2}{3}}
 \B(\iint_{\tilde{Q}_{k} } |\nabla  v |^{2}\Big) +\wgr{C}
\B(\sup_{-r^{2 }_{k}\leq t-t_{0}<0}\int_{\tilde{B}_{k}}|v|^{2}\B)^{5/3},
\ea$$
and
$$
\wgr{r_{k+1}}^{-\f{3}{2}}\|v\|_{L^{\f{5+2\delta}{1+2\delta}}
L^{\f{6(5+2\delta)}{11-2\delta}}(Q_{k+1})}\leq
 C\B( \f{1}{r^{3}_{k}}\sup_{-r^{2 }_{k}\leq t-t_{0}<0}\int_{\tilde{B}_{k}}|v |^{2}\B)^{\f{1}{2}}+\wgr{C}\B(r_{k}^{-3}\iint_{\tilde{Q} _{k}}
|\nabla v |^{2}\B)^{\f{1}{2}}
$$
Hence, there holds
\be\label{last2}\ba
\f{1}{r^{5 }_{k+1}}\iint_{\tilde{Q}_{k+1}}
|v|^{\f{10}{3}}\leq&  C\B( \f{1}{r^{3}_{k}}\sup_{-r^{2 }_{k}\leq t-t_{0}<0}\int_{\tilde{B}_{k}}|v |^{2}\B)^{\f53}\\&
+C\B( \f{1}{r^{3}_{k}}\sup_{-r^{2 }_{k}\leq t-t_{0}<0}\int_{\tilde{B}_{k}}|v |^{2}\B)^{\f32}\B(r_{k}^{-3}\iint_{\tilde{Q} _{k}}
|\nabla v |^{2}\B) \\
\leq& C  \varepsilon_{1}^{\f{5}{6}},
\ea
\ee and
\be\ba
r_{k+1}^{-\f{3(5+2\delta)}{2(1+2\delta)}}\|v\|^{\f{5+2\delta}{1+2\delta}}_{L^{\f{5+2\delta}{1+2\delta}}
L^{\f{6(5+2\delta)}{11-2\delta}}(Q_{k+1})}&\leq
 C\B( \f{1}{r^{3}_{k}}\sup_{-r^{2 }_{k}\leq t-t_{0}<0}\int_{\tilde{B}_{k}}|v |^{2}\B)^{\f{1}{2}\cdot \f{5+2\delta}{1+2\delta} }+\wgr{C}\B(r_{k}^{-3}\iint_{\tilde{Q} _{k}}
|\nabla v |^{2}\B)^{\f{1}{2}\cdot \f{5+2\delta}{1+2\delta}}\\
&\leq C\varepsilon^{\f{1}{2}\cdot\f{1}{2}\cdot\f{5+2\delta}{1+2\delta}}
\\&\leq C\varepsilon^{\f{5+2\delta}{4(1+2\delta)}}.\label{2.4.29}
\ea\ee
Putting together \eqref{last2}  and  \eqref{2.4.29}, we have
$$\iint_{\tilde{Q}_{k+1}} \!\!\!\!\!\!\!\!\!\! \! \!\!\!\!\!\!\!\!-\hspace{-0.16cm}-\hspace{-0.15cm}~~\,~
|v|^{\f{10}{3} }+r_{k+1}^{-\f{3(5+2\delta)}{2(1+2\delta)}}\|v\|^{\f{5+2\delta}{1+2\delta}}_{L^{\f{5+2\delta}{1+2\delta}}
L^{\f{6(5+2\delta)}{11-2\delta}}(Q_{k+1})}\leq  \wgr{\varepsilon}^{2/3}. $$
This completes the proof of this theorem.
\end{proof}
 
\section{\wgr{Improvement on} Caffarelli--Kohn--Nirenberg theorem by a logarithmic factor}
		\label{sec6}
		\setcounter{section}{6}\setcounter{equation}{0}
The goal of this section is to prove Theorem  \ref{the1.5}. According to   Proposition  \ref{add} in \wgr{Section 2}, the key point for the  extention      of $\sigma$ in \wgr{Theorem  \ref{the1.5}} is to    develop  \wgr{\cite[Lemma  4.2, p.818]{[RWW18]}}. Therefore, it suffices to present an improvement of  \wgr{this lemma}.
 To make the paper more readable, we will apply \eqref{wwz} for $\delta= 0$ to show
Theorem \ref{the1.5}.

\wgr{First, we give the definition of generalized parabolic Hausdorff measure as follows.
\begin{definition}
Let $h$ be an increasing continuous function on $(0, 1]$ with $\lim\limits_{r\rightarrow 0} h(r)=0$ and $h(1)=1$.
For fixed parameter $\delta>0$ and set $E\subset \R^3 \times \R$, we denote by $D(\delta)$
the family of all coverings $\{Q(x_i,t_i;\,r_{i})\}$ of $E$ with $0<r_{i}\leq \delta$. We denote
$$ \Psi_\delta (E, h)=\inf_{D(\delta)}\sum_i h(r_i)$$
and define the generalized parabolic Hausdorff measure as
$$\Lambda(E,h)=\lim_{\delta\rightarrow 0}\Psi_\delta (E, h).$$
\end{definition}
}

	In what follows, we set $m(r)=(\Gamma(r))^{\sigma}=(\log(e/r))^{\sigma}$, where $\sigma\in(0,1)$ will be determined later.

Before going further, we set
  $$
  F(m)=\B\{(x,t)\B |\limsup_{r\rightarrow0}\f{\wgr{E_{\ast}(r)}}{m(r)}\leq1 \B\}.
  $$

  In addition, we need the following fact due to  Choe and    Lewis \cite{[CL]}
  \begin{lemma}\cite{[CL]}
Assume that $(x,t)\in F(m)\cap\mathcal{S}$ . Then, there exists a positive constant $c_1$	independent of $(x,t)$
 such that
\begin{align}
&\limsup_{r\rightarrow 0}\f{ \wgr{E( r)}}{m^2(r)}\leq c_1.\label{CL}
\end{align}\end{lemma}
As said above, it suffices to prove the following lemma.	
		\begin{lemma}\label{J-b control}
Let  $(x,t)\in F(m)\cap\mathcal{S}$.  Then,  there exists a positive constant $c_2$ independent of $(x,t)$ such that
$$\liminf_{r\rightarrow 0}\wgr{J_q(r )}  m(r)^{\tau} \geq c_2,$$
where $\tau=\f{35-14q}{4}$.
\end{lemma}
\begin{proof}
Assume that the statement fails, then, for any $\eta>0$, there exists a singular point $(x,t)$ and a sequence $r_n\rightarrow 0$ such that
\begin{equation}\label{eq6.2}
\wgr{J_q (r_n )}  m(r_n)^{\tau} < \eta.
\end{equation}
It follows from \wgr{Lemma 2.5 and} \eqref{CL} and    that, for $\theta_n<1/8$,
$$\ba
\wgr{E_{5/2}(\theta_n r_n)} \leq& C \theta_{n} ^{ 5/2}m^{5/2}(r_n)+C \theta_{n}^{-\f{25-10q}{4}}m(r_n)^{\f{5-2q}{2}} \wgr{J_q(r_n)} \\
\leq& C \B[ m(r_n) ^{\tau}\wgr{J_{q}(r_{n})}\B]^{\f{2}{ 7-2q }}
\\\leq& C \eta^{\f{2}{ 7-2q }}, \ea$$
where
$\theta_n=[m(r_n)^{-q}\wgr{J_q(r_n)}]^{ \frac{ 4}{35-10q} }$.
Note that  $\theta_n$ goes to $0$ as $n\rightarrow \infty$ by \wgr{\eqref{eq6.2}} .
Let $\rho_n=\theta_n r_n$ and $\varepsilon_2=C \eta^{\f{2}{ 7-2q }}$ such that  $\varepsilon_2<\min\{1,\wgr{\varepsilon/2}\}$.
For
sufficiently large $n$, we see that
$$\wgr{E_{5/2}(\theta_n r_n)} \leq \varepsilon_2.$$
This \wgr{together with \eqref{wwz}}  implies that  $(x,t)$ is a regular point. Thus, we reach a contradiction and finish the proof.
\end{proof}

\section*{Acknowledgement}
		The research of Wang was partially supported by  the National Natural		Science Foundation of China under grant No. 11601492 and the
the Youth Core Teachers Foundation of Zhengzhou University of
Light Industry.
\wgr{The research of Wu was partially supported by the National
Natural Science Foundation of China under grant No. 11771423 and No. 11671378.}
			The research of Zhou	is supported in part by the National Natural Science
Foundation of China under grant No. 11401176  and Doctor Fund of Henan Polytechnic University (No. B2012-110).


\begin{thebibliography}{00}
			
			
			\bibitem{[CKN]}
			L. Caffarelli,  R. Kohn and L. Nirenberg,  Partial regularity
			of suitable weak solutions of the Navier-Stokes equations, {\it Comm. Pure. Appl. Math., }   \textbf{35} (1982), 771--831.



\bibitem{[CW17]}
 D. Chae and J. Wolf, Removing discretely self-similar singularities for the 3D Navier-Stokes equations.  {\it Comm. Partial Differential Equations}. \textbf{  42} (2017),   1359--1374.
\bibitem{[CW]}
\bysame, On the Liouville type theorems for self-similar solutions to the Navier-Stokes equations. {\it Arch. Ration. Mech. Anal.} \textbf{225} (2017),  549--572.
	





\bibitem{[CL]}
H.  Choe and J.  Lewis, On the singular set in the Navier-Stokes equations, {\it J. Funct. Anal.} \textbf{175}   (2000) 348--369.

\bibitem{[CY1]}
H.  Choe and M. Yang, Hausdorff measure of the singular set in the incompressible magnetohydrodynamic equations, {\it Comm. Math. Phys.} \textbf{336}   (2015) 171--198.





\bibitem{[CY2]}
\bysame, Hausdorff measure of boundary singular points in the magnetohydrodynamic equations. {\it  J. Differential Equations.} \textbf{260} (2016),  3380--3396.
			
	 
	



		\bibitem{[ESS]}
 L. Escauriaza,  G. Seregin and V. \v{S}ver\'{a}k,
 On $L^{\infty,3}$ -solutions to the Navier-tokes equations and Backward uniqueness,
 {\it Russian Mathematical Surveys,}
\textbf{58} (2003), 211--250.	




\bibitem{[Falconer]}
			K. Falconer, Fractal Geometry: Mathematical
	Foundations and Applications (New York: Wiley) 1990.

\bibitem{[GSS]}
G. Galdi, C. Simader, and H. Sohr, On the Stokes problem in Lipschitz
domains, Annali di Mat. pura ed appl. (IV), 167 (1994), pp. 147--163.

	
	\bibitem{[Giaquinta]}		M. Giaquinta, Multiple integrals in the calculus of variations and nonlinear elliptic systems. Annals of Mathematics Studies, 105. Princeton University Press, Princeton, NJ, 1983.

\bibitem{[GP]}
			C. Guevara   and N. C. Phuc,
			Local energy bounds and $\varepsilon$-regularity criteria for the 3D Navier-Stokes system. {\it Calc. Var.,} (2017) 56:68.	

\bibitem{[GKT]}
			S. Gustafson, K. Kang and T. Tsai,
			Interior regularity criteria for suitable weak
			solutions of the Navier-Stokes equations,
			{\it Commun. Math. Phys.} \textbf{273}  (2007),  161--176.



\bibitem{[HWZ]}C. He, Y. Wang and D. \wgr{Zhou},
New $\varepsilon$-regularity criteria and  application to the box  dimension of   the singular set
		in  the 3D Navier-Stokes equations, arxiv: 1709.01382.








	
		\bibitem{[JS]}H. Jia, V. \v{S}ver\'{a}k,
Local-in-space estimates near initial time for weak solutions of the Navier-Stokes equations and forward self-similar solutions. {\it Invent. Math.,} \textbf{196} (2014),   233--265.
			
\bibitem{[JWZ]}Q. Jiu, Y. Wang and D. Zhou,
On Wolf's regularity criterion of suitable weak solutions to the Navier-Stokes equations. arXiv:1805.04841






\bibitem{[KY16]}
Y. Koh and M. Yang,
The Minkowski dimension of interior singular points in
the incompressible Navier-Stokes equations. {\it J. Differential Equations.,} \textbf{261} (2016), 3137--3148.
\bibitem{[Kukavica1]}I. Kukavica,
Regularity for the Navier-Stokes equations with a solution in a Morrey space.   {\it
Indiana Univ. Math. J.} \textbf{57} (2008),  2843--2860.
\bibitem{[Kukavica]}
I. Kukavica, The fractal dimension of the singular set for solutions of the Navier-Stokes system.
{\it Nonlinearity.,} \textbf{22} (2009), 2889--2900.
\bibitem{[KP]}I.
Kukavica and Y. Pei, An estimate on the parabolic fractal dimension of the singular set for solutions of the Navier-Stokes system. {\it Nonlinearity.,} \textbf{25} (2012), 2775--2783.


			
			
			
			



\bibitem{[LS]}O.  Ladyzenskaja and G.  Seregin,  On
			partial regularity of suitable weak solutions to the three-dimensional Navier-Stokes equations, {\it   J. Math.
				Fluid Mech.,}  \textbf{1}   (1999),  356--387.
						
		\bibitem{[Lin]}F. Lin,   A new proof of the Caffarelli-Kohn-Nirenberg
			Theorem,    {\it Comm. Pure Appl. Math.,}   \textbf{51} (1998),  241--257.
	



			
		
			

\bibitem{[NRS]}
J. Ne\v{c}as,  M. R\r{a}u\v{z}i\v{c}ka and V. \v{S}ver\'{a}k, On Leray's self-similar solutions of the Navier-Stokes equations.{\it Acta. Math.,} \textbf{176} (1996),   283--294.
 	
\bibitem{[RWW]}
			W. Ren, Y. Wang and G. Wu,
		  Gang, Partial regularity of suitable weak solutions to the multi-dimensional generalized magnetohydrodynamics equations. {\it Commun. Contemp. Math.} \textbf{18} (2016),   1650018, 38 pp.
			
			
\bibitem{[RWW18]}			
	\bysame, Remarks on the singular set of suitable weak solutions for the three-dimensional Navier-Stokes equations. {\it J. Math. Anal. Appl.} \textbf{467} (2018),  807-824.		
			
			
			
		\bibitem{[RS2]}
J. Robinson and W. Sadowski,
On the dimension of the singular set of solutions to the Navier-Stokes equations, {\it Comm. Math.
Phys.,}
 \textbf{ 309} (2012), 497--506.		
		\bibitem{[SS18]}
 G. Seregin and V. \v{S}ver\'{a}k, Regularity Criteria for Navier-Stokes Solutions. Handbook of Mathematical Analysis in Mechanics of Viscous Fluids, 2018: 829-867.	
			
			

		
			
 				
			
			
			
		
			
			
			
			
		\bibitem{[TY]}	
	L. Tang and Y. Yu, Partial regularity of suitable weak solutions to the fractional Navier-Stokes equations, {\it Comm. Math. Phys.} \textbf{334} (2015) 1455--1482.		
			
\bibitem{[TX]}
			G. Tian and Z. Xin,
			Gradient estimation on Navier-Stokes equations,
		 {\it	Comm. Anal. Geom.}  \textbf{7}  (1999),  221-257.			
			
			
			
			
			
			
			
			
			
			
			
		 	
			
			
			
			
			
			
		



   \bibitem{[WW2]}Y. Wang  and G. Wu,
On the box-counting dimension of potential singular set for suitable weak solutions to the 3D Navier-Stokes equations,
  {\it Nonlinearity.,} \textbf{30} (2017), 1762-1772.
  

\bibitem{[WW18]}\bysame,
On local regularity criteria for suitable weak solutions to the 3D Navier-Stokes equations via the velocity gradient.  In preparetion.
  
  

\bibitem{[Wolf10]}
Wolf, J.: A new criterion for partial regularity of suitable weak solutions to the Navier-Stokes equations. Adv. Math. Fluid Mech.. Rannacher, A.S.R. ed., pp. 613--630. Springer (2010)
\bibitem{[W15]}
Wolf, J.: On the local regularity of suitable weak solutions to the generalized Navier-Stokes equations. {\it Ann. Univ. Ferrara},  \textbf{61} (2015),  149--171.






\end{thebibliography}
\end{document}